\documentclass[12pt]{amsart}
\usepackage{amsmath}
\usepackage{amsthm}
\usepackage{amssymb}
\usepackage{a4wide}
\usepackage{hyperref}
\usepackage{amsfonts}
\newtheorem{lemma}{Lemma}[section]
\newtheorem{corollary}{Corollary}[section]

\newtheorem{theorem}{Theorem}[section]

\newtheorem{remark}{Remark}[section]

\def\rr{\mathbb{R}}

\def\cn{\mathbb{C}}

\def\rr{\mathbb{R}}

\def\eps{\varepsilon}

\def\a{\alpha} 
\def\b{\beta} 
\def\c{\gamma}

 \def\e{\varepsilon}

\def\RR{{\mathbb R}}\def\QQ{{\mathbb Q}}
\def\NN{{\mathbb N}}
\def\var{{\rm Var}}
\DeclareMathOperator\arctanh{arctanh}

\title[Dispersion for BV coefficients]{Dispersion for 1-D Schr\"odinger and wave equations with BV coefficients}%\thanks{The authors have been partially supported by
\author[C. N. Beli, L. I. Ignat and E.  Zuazua]
{Constantin N. Beli, Liviu I. Ignat \and Enrique Zuazua}

\address{N. Beli
\hfill\break\indent Institute of Mathematics ``Simion Stoilow'' of the Romanian Academy\\21 Calea Grivitei Street \\010702 Bucharest \\ Romania }

\address{L. I. Ignat
\hfill\break\indent Institute of Mathematics ``Simion Stoilow'' of the Romanian Academy\\
\hfill\break\indent  21 Calea Grivitei Street \\010702 Bucharest \\ Romania 
\hfill\break\indent \and
\hfill\break\indent Faculty of Mathematics and Computer Science, University of Bucharest \\
\hfill\break\indent 14 Academiei Street, 010014 Bucharest, Romania.
}
 \email{{\tt
liviu.ignat@gmail.com}\hfill\break\indent  {\it Web page: }{\tt
http://www.imar.ro/\~\,lignat}}

\address{E. Zuazua
\hfill\break\indent  BCAM - Basque Center for Applied Mathematics\\
Alameda de Mazarredo 14. 
48009 Bilbao, Basque Country, Spain.
\hfill\break\indent \and
\hfill\break\indent  Ikerbasque, Basque Foundation for Science\\
Maria Diaz de Haro 3, 48013, Bilbao, Basque Country, Spain.
 }
  \email{{\tt zuazua@bcamath.org }
\hfill\break\indent {\it Web page: }{\tt http://www.bcamath.org/zuazua/}}


\begin{document}

%\linenumbers

\begin{abstract}In this paper we analyze the dispersion for one dimensional wave and Schr\"odinger equations with BV coefficients. In the case of the wave equation we give a complete answer in terms of the variation of the logarithm of the 
coefficient showing that dispersion occurs if this variation is small enough but it may fail when the variation goes beyond a sharp threshold. For the Schr\"odigner equation we prove that the dispersion holds under the same smallness assumption on the variation of the coefficient. But, whether dispersion may fail for larger coefficients is unknown for the Schr\"odinger equation.
\end{abstract}

\maketitle

\section{Introduction}

In this paper we consider the following two equations with variable coefficients: The one-dimensional wave equation
\begin{equation}\label{wave}
\left\{
\begin{array}{ll}
v_{tt}(t,x)-\partial_x(a(x)\partial_x v)(t,x)=0,& (t,x)\in \rr^2,\\[10pt]
v(0,x)=v_0(x),\ v_t(0,x)=0, & x\in \rr,
\end{array}
\right.
\end{equation}
and the Schr\"odinger equation 
\begin{equation}\label{sch}
\left\{
\begin{array}{ll}
iu_t(t,x)+\partial_x (a(x)\partial _xu)(t,x)=0,& (t,x)\in \rr^2,\\[10pt]
u(0,x)=u_0(x), & x\in \rr.
\end{array}
\right.
\end{equation}

Along the paper we will consider nonnegative functions $a$ with bounded variation and satisfying the following lower and upper bounds
\begin{equation}\label{cond.a}
0<m\leq a(x)\leq M, \quad x\in \rr.
\end{equation}

The main positive results of this paper are as follows.

\begin{theorem}\label{disp.wave}
For any  $a\in BV(\rr)$ satisfying \eqref{cond.a} and $\var(\log(a))<2 \pi$
there exists a positive constant $C(\var(a),m,M)$ 
such that the solution $v$ of \eqref{wave} satisfies
   \begin{equation}\label{dispersie.wave}
\sup _{x\in \rr} \int _{\rr} |v(t,x)|dt\leq {C(\var(a),m,M)}\|v_0\|_{L^1(\rr)}.
\end{equation}
\end{theorem}

\begin{theorem}\label{disp.sch}
For any  $a\in BV(\rr)$ satisfying \eqref{cond.a} and $\var(\log(a))<2 \pi$
there exists a positive constant $C(Var(a),m,M)$ such that the solution $u$ of \eqref{sch} satisfies 
   \begin{equation}\label{dispersie}
\|u(t)\|_{L^\infty(\rr)}\leq \frac {C(\var(a),m,M)}{\sqrt{t}}\|u_0\|_{L^1(\rr)}.
\end{equation}
\end{theorem}

In the case of the wave equation a counterexample can also be established when the total variation of the logarithm of the coefficient is large, showing that our dispersion result above is sharp.

\begin{theorem}\label{no-disp}Let be $0<m<M$ and $\alpha\geq 2\pi$.
For any positive number $N$ large enough
there exists a piecewise constant function, $m \leq a\leq M$, with $\var(\log (a))=\alpha$ such that for some $v_0\in L^1(\rr)$ the solution $v$ of problem \eqref{wave} satisfies
\begin{equation}\label{raport.N}
\sup _{x\in \rr}  \int _{\rr} |v(t,x)|dt\geq N \|v_0\|_{L^1(\rr)}.
\end{equation}
\end{theorem}

Such a counterexample is not available for the Schr\"odinger equation. Thus, whether the  dispersion result in Theorem \ref{disp.sch} is sharp  is an open problem.

Our results are given in terms of the total variation of  function $\log(a)$. However under the boundedness assumption above \eqref{cond.a}, $\var(a)$ and $\var(\log a)$ are comparable.
%$$\frac {\var (a)}M\leq \var (\log a)\leq \frac {\var (a)}m.$$

%
%The condition above the smallness of $\var(\log(a))$ is merely technical and we belive that the results are still true for all functions with bounded variation satisfying \eqref{cond.a}. 
The main  ideas of the proofs of the above results  come from the analysis of wave propagation in multi-layer structures \cite[Ch. 3]{MR2327824} and \cite{MR2049025}. The proof follows mainly the ideas in \cite{MR2049025} but with  finer resolvent estimates.

We recall that, once the dispersion is established for the solutions of the linear Schr\"odinger equation, more general space-time estimates can be obtained, namely,  the so-called Strichartz estimates
$$\|u\|_{L^q(\rr,L^r(\rr))}\leq C(q,r)\|\varphi\|_{L^2(\rr)},$$ 
for some admissible pairs $(q,r)$.  Strichartz estimates for BV coefficients in $1$-d without smallness conditions have been established in 
\cite{MR2227135} without making use of the dispersion property. This paper is devoted to investigate under which assumptions the dispersion property still holds. 

Estimates similar to these in Theorem \ref{disp.wave} but integrating on the space variable $x$ instead of time, have been obtained in \cite{MR2794904} under a smallness assumption on the BV-norm of  $\log(a)$. The methods developed in this paper could very likely be useful to further analyze the problems addressed in \cite{MR2794904}. But this is still to be done.

The paper is organized as follows. In section \ref{resolvent} we present some preliminary results from \cite{MR2049025} and state  two technical lemmas that allow us to improve the results in \cite{MR2049025}. In section \ref{main.section} we prove the main results stated in the introduction. We point out that  the proof of Theorem \ref{disp.sch}  uses previous results from the proof of  Theorem \ref{disp.wave}.
 Section   \ref{proof.lemmas} contains the proofs of the two technical lemmas. We will obtain  estimates on some almost periodic functions by using some tools from analytical number theory.
\section{Resolvent estimates  on a laminar media}\label{resolvent}

In this section we 
collect some previous results  from \cite{MR2049025}, keeping the same notations. 
Let us consider a partition of the real axis 
\begin{equation}\label{partition}
-\infty=x_0<x_1<x_2<\cdots<x_{n-1}<x_n=\infty
\end{equation}
and a step function 
\begin{equation}\label{function.a}
a(x)=b_k^{-2},\quad x\in I_k =(x_{k-1},x_k), \, k=1,\ldots, n,
\end{equation}
where $1/M^2\leq b_k\leq 1/m^2$.

Let us now consider the self-adjoint operator $A=-\partial_xa(x)\partial _x$ defined from $\{h\in H^1(\rr), a\partial _x h\in L^2(\rr)\}$ to $L^2(\rr)$. For $\omega\geq 0$ let us consider $R_\omega$ its resolvent:
$$R_\omega g=(-\partial _x a(x)\partial _x+\omega^2 I)^{-1}g.$$
It follows that for $x\in I_k=(x_{k-1},x_k)$, $k=1,\ldots, n$, we have
\begin{equation}\label{res}
R_\omega g(x)=c_{2k-1}(\omega)e^{\omega b_kx}+c_{2k}(\omega)e^{-\omega b_k x}+b_k\int _{I_k} \frac{g(y)}{2\omega}e^{-\omega b_k|x-y|}dy,
\end{equation}
where $c_2=c_{2n-1}=0$ and the other coefficients are determined by solving the system obtained from the continuity of $R_\omega g$ and $a(x)\partial _x R_\omega g$ at the points $x_k$, $k=1,\ldots, n-1$. It follows that
$$D_n(\omega) C =T$$
where $C=[c_1,c_3,c_4,\ldots,c_{2n-3},c_{2n-2},c_{2n}]^{\rm T}$, $T=(t_1,\ldots, t_{n-1})^{\rm T}$, $t_k=(t_{k1},t_{k2})^{\rm T}$,

$$t_k=\left(\begin{array}{c}
\displaystyle -b_k\int _{I_k} \frac{g(y)}{2\omega}e^{-\omega b_k(x_k-y)}dy + b_{k+1}\int _{I_{k+1}} \frac{g(y)}{2\omega}e^{-\omega b_{k+1}(y-x_k)}dy  \\[10pt]
\displaystyle b_kb_{k+1} \Big(\int _{I_k} \frac{g(y)}{2\omega}e^{-\omega b_k(x_k-y)}dy + \int _{I_{k+1}} \frac{g(y)}{2\omega}e^{-\omega b_{k+1}(y-x_k)}dy\Big)
\end{array}\right)
,\ k=1,\ldots,n-1,$$
and $$D_n=\left(\begin{array}{cccccccc}{\bf a}_1 & B_1 & 0 & 0 & 0 & 0 & 0 & 0 \\0 & A_2 & B_2 & 0 & 0 & 0 & 0 & 0 \\0 & 0 & A_3 & B_3 & 0 & 0 & 0 & 0 \\- & - & - & - & - & - & - & - \\0 & 0 & 0 & 0 & A_{n-3} & B_{n-3} & 0 & 0 \\0 & 0 & 0 & 0 & 0 & A_{n-2} & B_{n-2} & 0 \\0 & 0 & 0 & 0 & 0 & 0 & A_{n-1} &{\bf b}_{n-1}\end{array}\right)
$$
with 
$${\bf a}_1=
\left(\begin{array}{c} 
e^{\omega b_1x_1} \\
b_2 e^{\omega b_1x_1} 
\end{array}\right), \quad 
{\bf b}_{n-1}=\left(\begin{array}{c} 
-e^{-\omega b_nx_{n-1}} \\
b_{n-1} e^{-\omega b_nx_{n-1}} 
\end{array}\right),
$$
$$A_k=\left(\begin{array}{cc}
e^{\omega b_kx_k} & e^{-\omega b_kx_k}  \\
b_{k+1}e^{\omega b_kx_k}  & -b_{k+1}e^{-\omega b_{k}x_k}
\end{array}\right)
,\ B_k=\left(\begin{array}{cc}
-e^{\omega b_{k+1}x_k} & -e^{-\omega b_{k+1}x_k}  \\
-b_{k}e^{\omega b_{k+1}x_k}  & b_{k}e^{-\omega b_{k+1}x_k}
\end{array}\right).
$$
For technical reasons  we introduce the matrix $\tilde D_n$ which has the same structure as $D_n$ but replacing vector ${\bf b}_{n-1}$ with 
$$\tilde {\bf b}_{n-1}=\left(\begin{array}{c} 
-e^{\omega b_nx_{n-1}} \\
-b_{n-1} e^{\omega b_nx_{n-1}} 
\end{array}\right).
$$
We point out that the vectors ${\bf b}_{n-1}, \tilde {\bf b}_{n-1} $  appearing in $D_n$ and $\tilde D_n$ are given by the second and respectively first column of $B_{n-1}$.

%Tacking into account that $I_k=(x_{k-1},x_k)$ we can write $t_k$ as
%$$t_k=\left(\begin{array}{c}
%\displaystyle -b_k\int _{I_k} \frac{g(y)}{2\omega}e^{-\omega b_k(x_k-y)}dy + b_{k+1}\int _{I_{k+1}} \frac{g(y)}{2\omega}e^{-\omega b_{k+1}(y-x_k)}dy  \\[10pt]
%\displaystyle b_kb_{k+1} \Big(\int _{I_k} \frac{g(y)}{2\omega}e^{-\omega b_k(x_k-y)}dy + \int _{I_{k+1}} \frac{g(y)}{2\omega}e^{-\omega b_{k+1}(y-x_k)}dy\Big)
%\end{array}\right).
%$$

Let us introduce now the reflection coefficients 
\begin{equation}\label{reflex}
d_{k-1}=\frac{b_{k-1}-b_k}{b_{k-1}+b_k}, \, k=2,\ldots,n
\end{equation}
and the functions $Q_k$, $k=1,\ldots,n$, defined as follows: $Q_1(\omega)\equiv 0$ and 
\begin{equation}\label{Q}
Q_k(\omega)=
\left\{
\begin{array}{ll}
\displaystyle e^{-2\omega b_k(x_k-x_{k-1})} \frac{-d_{k-1}+Q_{k-1}(\omega)}{1-d_{k-1}Q_{k-1}(\omega)}, &k=2,\ldots,n-1,\\[12pt]
\displaystyle e^{2\omega b_kx_{k-1}} \frac{-d_{k-1}+Q_{k-1}(\omega)}{1-d_{k-1}Q_{k-1}(\omega)}, &k=n.
\end{array}
\right.
\end{equation}
It follows that for $2\leq k\leq n$ 
$$Q_k(\omega)=
\left\{
\begin{array}{ll}
\displaystyle e^{-2\omega b_k x_k}\frac {\det \tilde D_k(\omega)}{\det D_k(\omega)},& k=2,\ldots,n-1,\\[12pt]
\displaystyle \frac {\det \tilde D_k(\omega)}{\det D_k(\omega)},& k=n,
\end{array}
\right. 
$$ and for any $2\leq k\leq n$
\begin{equation}\label{Dn}
\det D_k(\omega)=\prod _{j=1}^{k-1} (b_{j}+b_{j+1})e^{\omega (b_j-b_{j+1})x_j}(1-d_jQ_j(\omega)).
\end{equation}
It has been proved in \cite{MR2049025} that there exists a $\delta>0$ such that for any $\omega\in \cn$ with $\Re(\omega)>-\delta$ we have $|Q_k(\omega)|<1$, $k=2,\ldots, n$. It implies that $(\det D_n(\omega) )^{-1}$ is uniformly bounded in the same region of the complex plane and moreover $\omega R_\omega u_0$  can be analytically continued. In the case when the coefficient $a$ is as in \eqref{function.a}  the spectral calculus gives us the following representation of the solutions  of equations  \eqref{wave} and \eqref{sch}.
\begin{lemma}\label{spectral.wave0}
The solution of the wave equation \eqref{wave} verifies
\begin{equation}\label{spectral.wave}
v(t,x){\bf 1}_{\{t>0\}}=\frac {i}{2\pi}\int _{-\infty}^{\infty} e^{it\omega } \omega (R_{i\omega }v_0)(x){ d\omega}.
\end{equation}
\end{lemma}
\begin{lemma}\label{spectral.sch0}
The solution of the Schr\"odinger equation \eqref{sch} verifies
\begin{equation}\label{spectral.sch}
u(t,x)=\frac 1{ i \pi} \int _{-\infty}^{\infty} e^{-it\omega^2 } \omega (R_{i\omega }u_0)(x){d\omega}.
\end{equation}
\end{lemma}

For completeness we prove these lemmas. 

\begin{proof}[Proof of Lemma \ref{spectral.wave0}]
Set $v_1(t)=v(t){\bf 1}_{\{t>0\}}$. It follows that $v_1$ satisfies 
$$\partial_{tt}v_1-\partial _x(a(x)\partial _x v_1)=\partial _t \delta_0 f.$$
Since $v_1(t)$ is supported on $(0,\infty)$ it follows that the Fourier transform in time variable of $v_1$ is holomorphic in the domain 
$\{\Im z<0\}$ and  verifies the equation
$$(-z^2-\partial _x(a(x)\partial _x )\hat v_1(z,\cdot)=iz f(\cdot), \quad \Im(z)<0.$$
Taking $z=\omega -i \eps$, $\omega\in \rr$, $\eps>0$ small enough we obtain that 
$$\hat v_1(\omega-i\eps)= i(\omega-i\eps) R_{i(\omega-i\eps)}f.$$
Using the inverse Fourier transform we get
$$v_1(t)= \frac 1{2\pi}\int _\rr e^{it\omega} i(\omega-i \eps)R_{i(\omega-i\eps)}fd\omega.$$
Since $\omega R_{\omega}$ can be analytically continued on $\{\Re z>-\delta\}$ we obtain the desired result.

A similar argument shows that 
$$v(t){\bf 1}_{\{t<0\}}=-\frac 1{2\pi}\int _\rr e^{it\omega} i(\omega+i \eps)R_{i(\omega+i\eps)}fd\omega.$$
The proof is now complete.
\end{proof}

\begin{proof}[Proof of Lemma \ref{spectral.sch0}]
Using the identity 
\begin{align*}
h(\lambda)&=\lim _{\eps \downarrow 0} \frac 1{2\pi i} \int _{\rr} h(s)\Big(\frac 1{\lambda -(s-i\eps)}-\frac 1{\lambda -(s+i\eps)}\Big)ds\\
&=\frac 1{2\pi i} \int _{\rr} h(s)\Big(\frac 1{\lambda -(s-i0)}-\frac 1{\lambda -(s+i0)}\Big)ds,
\end{align*}
classical spectral calculus gives us that
$$u(t,x)=e^{-itA}=\frac 1{2\pi i}\int _{\rr} e^{-its}\Big( (A-(s-i0))^{-1}-(A-(s+i0))^{-1}\Big)u_0 ds .$$
Since $\sigma(A)=(0,\infty)$ we have that $(A+z)^{-1}$ is analytic on $\cn \setminus (-\infty,0)$ and that
$$(A+(-\tau^2+i0))^{-1}=R_{i|\tau|+0}, \quad (A+(-\tau^2-i0))^{-1}=R_{-i|\tau|+0}.$$
Then
\begin{align*}
u(t,x)=&\frac 1{2\pi i}\int _{0}^\infty  e^{-its} \Big( (A-(s-i0))^{-1}-(A-(s+i0))^{-1}\Big)u_0 ds\\
&=\frac 1{\pi i}\int _{0}^\infty  e^{-it\tau ^2}\tau \Big( (A-(\tau^2-i0))^{-1}-(A-(\tau^2+i0))^{-1}\Big)u_0  d\tau \\
&=\frac 1{\pi i}\int _{0}^\infty  e^{-it\tau ^2} \tau\Big( R_{i\tau +0}u_0-R _{-i\tau+0}u_0\Big)  d\tau\\
&=\frac 1{\pi i}\int _{-\infty}^\infty  e^{-it\tau ^2} \tau  R_{i\tau +0}u_0  d\tau.
\end{align*}
Using now that $\omega R_{\omega}$ can be analytically continued on $\{\Re z>-\delta\}$ we obtain the desired result.
\end{proof}

In the following we denote by $\hat f$ and $f^\vee$ the Fourier and the inverse Fourier transform of the function $f$:
$$\hat f(\xi)=\int _{\rr}f(x)e^{-ix\xi}dx, \quad
f^\vee (x)=\frac 1{2\pi}\int _{\rr} f(\xi)e^{i\xi x}d\xi.$$

The proof of the main results of this paper requires the theory of almost periodic functions.
A function $f:\rr\rightarrow \cn$ is said to be almost-periodic if it can be represented as 
$$f(t)=\sum _{n}c_n e^{i\lambda_n t}, $$
and the following norm satisfies
$$\|f\|_{AP}=\sum _{n}|c_n|<\infty.$$
It is easy to see that the space of almost periodic functions is an algebra. For more details on the properties of these functions we refer to \cite{MR2460203}.

As observed in \cite{MR2049025}, the function $\det D_n(i\omega)$ is an almost periodic function. The same property is satisfied by $1/\det D_n(i\omega)$ even if this property is not trivial (see \cite{MR2049025}, section 2.2). 

Here, in addition to the results  in \cite{MR2049025}, we will compute exactly the coefficients $c_k$ in terms of vector $T$ 
and sequence $\{Q_k\}_{k=1}^n$ by solving the system $D_n(\omega)C=T$ (see Section \ref{main.section} below).
The argument  in \cite{MR2049025} only uses the fact that, since $C$ is a solution of the above system, then its components are finite sums of the terms in the vector $T$.
Also, instead of using the results in \cite[Section 2.2]{MR2049025} we control in a finer way  the sequence $\{Q_k\}_{k=1}^n$ introduced in \eqref{Q} and prove  the following two key lemmas.

\begin{lemma}\label{imp}
Let us consider two sequences of real numbers $(c_n)_{n\geq 1}$ and $(d_n)_{n\geq 1}$ with $ |d_n|\leq d<1$
satisfying 
\begin{equation}\label{small}
\sum _{n\geq 1}\arctanh |d_n|<\frac \pi 2.
\end{equation}
We also consider the following sequence of functions $Q_1(i\omega)\equiv 0$ and 
\begin{equation}\label{def.q}
Q_k(i \omega)=e^{ic_{k-1}\omega } \frac{-d_{k-1}+Q_{k-1}(i\omega)}{1-d_{k-1}Q_{k-1}(i\omega)}, \omega \in \rr, k\geq 2.
\end{equation}
Then for any $n\geq 2$ the following holds
\begin{equation}\label{est.ap}
\|Q_n\|_{AP}\leq \tan \Big(  \sum _{k\geq 1}\arctanh |d_k|\Big).
\end{equation}
\end{lemma}

Moreover, Lemma \ref{imp} is the best result possible in the following
sense:

\begin{lemma}\label{negative}
Let $\alpha >0$ and $0<d<1$. For any $\e>0$ and any $N>0$ there exist
a positive integer $n$ and a sequence $0<\tilde {d}_k\leq d$,
$k=1,\ldots,n-1$, with
$$\sum _{k=1}^{n-1}\arctanh \tilde d_k=\alpha
$$
such that  for any sequence  $\{c_k\}_{k= 1}^{n-1}$ of numbers
linearly independent over the rationals and any sequence $\{d_k\}_{k=
1}^{n-1}$ with $|d_k|=\tilde d_k$ the sequence $\{Q_k\}_{k=1}^n$ defined
in \eqref{def.q} satisfies
\begin{equation}\label{mare}
\|Q_n\|_{AP}\geq\tan\a -\e\quad\text{if }\a <\pi/2\quad\text{and}
\end{equation}
\begin{equation}\label{large}
\|Q_n\|_{AP}\geq N\quad\text{if }\a\geq\pi/2.
\end{equation}
\end{lemma}

\begin{remark}
Let us remark that the reflection coefficients  $d_k$  in
\eqref{reflex}  can be rewritten as
\begin{equation}\label{log}
d_k=\tanh  \Big( \frac {\log (b_{k-1})-\log (b_k)}2\Big).
\end{equation}
Since $|\tanh x|=\tanh |x|$ we have 
$$|d_k|=\tanh  \Big| \frac {\log (b_{k-1})-\log (b_k)}2\Big|.$$
As a consequence: 
\begin{equation}\label{indentity.bv}
\sum _{k\geq 1}\arctanh |d_k|=\sum_{k\geq 1}  \Big| \frac {\log (b_{k-1})-\log (b_k)}2\Big|=\frac{\var (\log (a))}4.
\end{equation}
Also, since $a$ satisfies \eqref{cond.a} we also have
\begin{equation}\label{220a}
\frac {\var (a)}M\leq \var (\log a)\leq \frac {\var (a)}m.
\end{equation}
\end{remark}

With  Lemma \ref{imp} and Lemma \ref{negative} we can prove Theorem \ref{disp.wave} and Theorem \ref{no-disp}. Theorem \ref{disp.sch} will be a consequence of Theorem \ref{disp.wave}.

Let us now comment on how these lemmas apply to obtain the main results in this paper.
The key point in the proof of Theorem \ref{disp.wave} is, as we will see in Section \ref{main.section},  that for a step function $a$ as in \eqref{partition} and \eqref{function.a} the following holds
\begin{equation}\label{key.1}
\displaystyle \sup_{v_0\in L^1(\rr)}\frac{\displaystyle\sup_{x\in \rr}\int_{\rr} |v(t,x)|dt}{\|v_0\|_{L^1(\rr)}}\simeq
\sup_{v_0\in L^1(\rr)}\frac{\displaystyle \|(Q_n(i\omega) \widehat v_0 )^\vee)\|_{L^1(\rr)}}{\|v_0\|_{L^1(\rr)}}=
\|Q_n\|_{AP}.
\end{equation}
Thus the results given by Lemma \ref{imp} and Lemma \ref{negative} on the $AP$-norm of $Q_n$ provide results for the behavior of the solutions of the wave equation \eqref{wave}.

In the case of the Schr\"odinger equation \eqref{sch}, using the same arguments as in the case of the wave equation,   we have that
\begin{equation}\label{key.2}
\sup_{u_0\in L^1(\rr)}\frac{t^{1/2}\|u(t)\|_{L^\infty(\rr)}}{\|u_0\|_{L^1(\rr)}}\simeq 
\sup_{u_0\in L^1(\rr)}\frac{t^{1/2}\|(e^{it\omega^2}Q_n(i\omega)\widehat u_0)^\vee \|_{L^\infty(\rr)}}{\|u_0\|_{L^1(\rr)}}.
\end{equation}
Applying Young's inequality it is immediate that for all $t>0$ the following holds
\begin{align*}\label{key.3}
\sup_{u_0\in L^1(\rr)} & \frac{t^{1/2}\|(e^{it\omega^2}Q_n(i\omega)\widehat u_0)^\vee \|_{L^\infty(\rr)}}{\|u_0\|_{L^1(\rr)}}\leq 
\sup_{v_0\in L^1(\rr)}\frac{\|(Q_n(i\omega)\widehat v_0)^\vee \|_{L^1(\rr)}}{\|v_0\|_{L^1(\rr)}}\leq \|Q_n\|_{AP}.
\end{align*}
However, we cannot say that the right hand side in \eqref{key.2} is comparable with $\|Q_n\|_{AP}$. This is why we have only a positive result  when the   $BV$-norm of the coefficient $a$ is small. The optimality of the result in Theorem \ref{disp.sch} is still an open problem.

%----------------------------
%
%maybe the next phrases are not needed here
%
%
%---------------------------
%
%
%On the other hand observe that 
%$$u(t)=K_t\ast (Q_n(i\omega) \hat u_0)^\vee.$$
%Then 
%\begin{align*}
%\frac{\|u(t)\|_{\infty}}{\|u_0\|_{L^1(\rr)}}&=\frac{\|K_t\ast (Q_n(i\omega) \hat u_0)^\vee\|_{L^\infty(\rr)}}{\|(Q_n(i\omega) \hat u_0)^\vee\|_{L^1(\rr)}}\frac{\|(Q_n(i\omega) \hat u_0)^\vee\|_{L^1(\rr)}}{\|u_0\|_{L^1(\rr)}}\\
%&\leq
%\frac 1{\sqrt t} \frac{\|(Q_n(i\omega) \hat u_0)^\vee\|_{L^1(\rr)}}{\|u_0\|_{L^1(\rr)}}\leq \frac{\|Q_n\|_{AP}}{\sqrt t}.
%\end{align*}
%and we are not able to construct maximizers for the above inequality in order to show that $\|Q_n\|_{AP}$ is the best constant that can appear.
%\section{Almost periodic functions and estimates on $Q_k$}

\section{Proof of the main results}\label{main.section}
\setcounter{equation}{0}

 The aim of this section is to prove the main results of this paper. 
We first concentrate on the case of the wave equation \eqref{wave}.
We will  prove that the solution of equation \eqref{wave} satisfies 
 \begin{equation}\label{dispersie.2}
\sup _{x\in \rr} \int _{\rr} |v(t,x)|dt\leq {C\Big(\var (a),m,M\Big)}\|v_0\|_{L^1(\rr)}.
\end{equation}

An argument similar to the one in \cite[Prop.~1.3]
{MR2227135} shows that it is sufficient to prove results for piecewise constant functions $a$ taking a finite number of values.
Let us consider a laminar medium as in Section \ref{resolvent}.
The key point in our proof is that the above estimate is equivalent with the fact that $Q_n(i\omega)$ is an $L^1(\rr)$-Fourier multiplier and the norm of $Q_n$ as a $L^1(\rr)$-multiplier can be estimated in terms of the variation of $\log (a)$.
%Moreover,
%we  point out that it is sufficient to consider $v_0$ to be supported in one of the intervals $I_k$, $k=1,\ldots,n$ since by linearity the result extends to any function $v_0\in L^1(\rr)$. 

%We will show that \eqref{dispersie.2} holds if the 
% sum of the reflection coefficients is small  and is false if it is large.

\begin{proof}[Proof of Theorem \ref{disp.wave}]
Using the spectral formula  \eqref{spectral.wave} and the representation of the resolvent $R_{i\omega}$ obtained in the previous section, for any $x\in I_k$ and $t>0$,   the solution $v$ of equation \eqref{wave} can be written as
\begin{align*}
v(t,x)&=\\
=&\int _{\rr} \Big(
i\omega   c_{2k-1}(i\omega)e^{i\omega(t+b_kx)}+i\omega c_{2k}(i\omega)e^{i\omega(t- b_k x)} \Big)\frac{d\omega}{2\pi}
+\frac{b_k}{4\pi}\int _{\rr} e^{it\omega} \int _{I_k} {v_0(y)}e^{-i\omega b_k|x-y|}{dy d\omega}\\
=&{(i\omega c_{2k-1}(i\omega))}^\vee(t+b_kx)+
{(i\omega c_{2k}(i\omega))}^\vee(t-b_kx)+\tilde v(t,x)
\end{align*}
where 
\begin{align*}
\tilde v(t,x)=&\frac{b_k}{4\pi}\int _{\rr} e^{it\omega} \int _{I_k} ({v_0} {\bf{1}}_{\{x_{k-1}<y<x)\}})(y)e^{-i\omega b_k (x-y)}{dy d\omega}\\
&+\frac{b_k}{4\pi}\int _{\rr} e^{it\omega} \int _{I_k} ({v_0} {\bf{1}}_{\{x<y<x_k)\}})(y)e^{-i\omega b_k (y-x)}{dy d\omega}\\
=&\frac 1{2} \int _{\rr} e^{i\omega (t/b_k -x)}  ({v_0} {\bf{1}}_{\{x_{k-1}<y<x)\}})^\vee(\omega)d\omega+\frac 1{2}
\int _\rr{e^{i\omega(t/b_k+x)}} ({v_0} {\bf{1}}_{\{x<y<x_k)\}})^{\wedge}(\omega)d\omega\\
=& \frac 12   ({v_0} {\bf{1}}_{\{x_{k-1}<y<x)\}})(x-\frac t{b_k})+\frac 12 ({v_0} {\bf{1}}_{\{x<y<x_k)\}})(x+\frac  t{b_k}).
\end{align*}
%The function $\tilde v$ corresponds to waves that  propagate freely  with velocity $1/b_k$ to the right or to the left until touch the boundary layers.

It is easy to see that
\begin{equation}\label{est.tilde}
\int _{\rr}|\tilde v(t,x)|dt\leq b_k \|v_0\|_{L^1(\rr)}\leq m^{-2}\|v_0\|_{L^1(\rr)}.
\end{equation}
Since for $t<0$, $v(t,x)=v(-t,x)$ we have 
$$\int _{\rr}|v(t,x)|dt=2\int_{0}^\infty |v(t,x)|dt.$$
Hence, in order to prove estimate \eqref{dispersie.2}, it remains to show that for any $j=1,\ldots,n$ and for any $x\in I_j$ the following holds:
\begin{equation}\label{main.ineq.vechi}
\int _{\rr}|{(i\omega c_{2j-1}(i\omega))}^\vee(t+b_jx)+
{(i\omega c_{2j}(i\omega))}^\vee(t-b_jx)|dt\leq{C\Big(\var (a),m,M\Big)}\|v_0\|_{L^1(\rr)}.
\end{equation}
In fact we will prove a stronger estimate: for any $1\leq j\leq n-1$ the following holds:
\begin{equation}\label{main.ineq}
\int _{\rr}\Big( |{(i\omega c_{2j-1}(i\omega))}^\vee(x)|+
|{(i\omega c_{2j}(i\omega))}^\vee(x)|\Big)dx\leq{C\Big(\var (a),m,M\Big)}\|v_0\|_{L^1(\rr)}.
\end{equation}
We also remark that since $c_{2}=c_{2n-1}=0$,  when $j\in \{1,n-1\}$,  the two estimates \eqref{main.ineq.vechi} and \eqref{main.ineq} are the same.
Estimate \eqref{main.ineq} is the key  not only in the proof of Theorem \ref{disp.wave} but also in the one of Theorem \ref{disp.sch}. Moreover,
we  point out that it is sufficient to consider $v_0$ to be supported in one of the intervals $I_k$, $k=1,\ldots,n$ since by linearity the result extends to any function $v_0\in L^1(\rr)$. 

Since there is a strong connection between the $c$'s and $t$'s we observe that for $k=1,\ldots, n-1$,
\begin{equation}\label{t1}
i\omega t_{k,1}(i\omega)=-\frac {b_k}2e^{-i\omega b_kx_k}\widehat{v_0{\bf 1}_{I_k}}(-b_k\omega)+\frac {b_{k+1}}2e^{-i\omega b_{k+1}x_{k}}\widehat{v_0{\bf 1}_{I_{k+1}}}(b_{k+1}\omega)
\end{equation}
and
\begin{equation}\label{t2}
i\omega t_{k,2}(i\omega)=\frac {b_kb_{k+1}}2\Big(
e^{-i\omega b_kx_k}\widehat{v_0{\bf 1}_{I_k}}(-b_k\omega)+e^{-i\omega b_{k+1}x_{k}}\widehat{v_0{\bf 1}_{I_{k+1}}}(b_{k+1}\omega)\Big).
\end{equation}
An immediate consequence is that
\begin{equation}\label{t.1}
\int _{\rr}|(i\omega t_{k,1}(i\omega))^\vee(x)|dx\leq \frac{|b_k|+|b_{k+1}|}2\int _{\rr}|v_0|\leq c(m)\int _{\rr}|v_0|
\end{equation}
and
\begin{equation}\label{t.2}
\int _{\rr}|(i\omega t_{k,2}(i\omega))^\vee(x)|dx\leq {|b_k||b_{k+1}|}\int _{\rr}|v_0|\leq c(m)\int _{\rr}|v_0|.
\end{equation}

The main steps in the proof of \eqref{main.ineq} are the following:
\begin{itemize}

\item Prove \eqref{main.ineq} for $j=n$ and ${\rm{supp}}\, v_0\subset I_n $.

\item Prove \eqref{main.ineq} for  $j=n$ and ${\rm{supp}}\, v_0\subset I_1 $.

 By symmetry the same holds for $j\in \{1,n\}$ and ${\rm{supp}}\, v_0\subset I_1 \cup I_n$.

\item  Prove \eqref{main.ineq} for $j=n$ and ${\rm{supp}}\, v_0\subset I_k $, $2\leq k\leq n-1$.

\item  Prove \eqref{main.ineq} for $j\in \{k,\ldots,n-1\}$ and ${\rm{supp}}\, v_0\subset I_k $ with  $2\leq k\leq n-1$.

 By symmetry the same holds for $x\in I_j$  and ${\rm{supp}}\, v_0\subset I_k $, $2\leq k\leq n-1$ with $1\leq j\leq k-1$.

\end{itemize}

\textbf{Case 1. Computing $c_{2n}$ when  ${\rm{supp}}\, v_0\subset I_n $. } In this case  all $t_{k}$, $k=1,\ldots, n-2$, vanish. Moreover $t_{n-1,2}=b_{n-1}t_{n-1,1}$. It follows that
\begin{align*}
c_{2n}(i\omega)&=(\det D_n(i\omega))^{-1}
\left|\begin{array}{cccccccc}{\bf a}_1 & B_1 & 0 & 0 & 0 & 0 & 0 & 0 \\0 & A_2 & B_2 & 0 & 0 & 0 & 0 & 0 \\0 & 0 & A_3 & B_3 & 0 & 0 & 0 & 0 \\- & - & - & - & - & - & - & - \\0 & 0 & 0 & 0 & A_{n-3} & B_{n-3} & 0 & 0 \\0 & 0 & 0 & 0 & 0 & A_{n-2} & B_{n-2} & 0 \\0 & 0 & 0 & 0 & 0 & 0 & A_{n-1} & t_{n-1}\end{array}\right|.
\end{align*}
Developing over the last two lines the above determinant we obtain 
\begin{align*}
c_{2n}(i\omega)&=
-\left|\begin{array}{cc}e^{i\omega b_{n-1}x_{n-1}} & t_{n-1,1} \\b_ne^{i\omega b_{n-1}x_{n-1}} & t_{n-1,2}\end{array}\right|
\frac{\det D_{n-1}(i\omega)}{\det D_n(i\omega)}
+\left|\begin{array}{cc}e^{-i\omega b_{n-1}x_{n-1}} & t_{n-1,1} \\-b_ne^{-i\omega b_{n-1}x_{n-1}} & t_{n-1,2}\end{array}\right| \frac{\det\tilde D_{n-1}(i\omega)}{\det D_n(i\omega)}
\\
&= -e^{i\omega b_{n-1}x_{n-1}} \left|\begin{array}{cc}1 & 1 \\b_n & b_{n-1}\end{array}\right|
t_{n-1,1} \frac{\det D_{n-1}(i\omega)}{\det D_n(i\omega)} \\
&\quad+e^{-i\omega b_{n-1}x_{n-1}} \left|\begin{array}{cc}1 & 1 \\-b_n & b_{n-1}\end{array}\right|
t_{n-1,1}\frac{\det \tilde D_{n-1}(i\omega) )}{\det D_n(i\omega)}\\
&=t_{n-1,1}
\Big( (b_{n-1}+b_n) e^{-i\omega b_{n-1}x_{n-1}}\frac{\det \tilde D_{n-1}(i\omega) )}{\det D_n(i\omega)}+(b_n-b_{n-1})e^{i\omega b_{n-1}x_{n-1}}\frac{\det  D_{n-1}(i\omega) )}{\det D_n(i\omega)}\Big)\\
&=-t_{n-1,1}e^{-i \omega b_n x_{n-1}}\frac {\det \tilde D_n(i\omega)}{\det D_n(i\omega)},
\end{align*}
where the last identity involving $\det \tilde D_n$, $\det \tilde D_{n-1}$ and $\det D_{n-1}$ has been proved in \cite[p.~871]{MR2049025}.
 
From \eqref{t1} we have that
$$i\omega t_{n-1,1}(i\omega)=\frac {b_n}2e^{-i\omega b_nx_{n-1}}\widehat{v_0I_n}(b_n\omega)$$
and then we obtain the exact formula of $c_{2n}(i\omega)$ in terms of $Q_n(i\omega)$:
\begin{equation}\label{c.2n}
i\omega c_{2n}(i\omega)=-\frac{b_n}2 e^{-2i \omega b_nx_{n-1}} \frac{\det \tilde D_n(i\omega)}{\det  D_n(i\omega)}=
-\frac{b_n}2 e^{-2i \omega b_nx_{n-1}} Q_n(i\omega)
\widehat{v_0I_n}(b_n\omega).
\end{equation}
This identity is the key point in  proving  not only Theorem \ref{disp.wave} but also Theorem \ref{no-disp}.

It follows that 
\begin{equation}\label{caso.1}
\int _{\rr} | {(i\omega c_{2n}(i\omega))}^\vee(x)|dx\leq \frac{b_n}2 \|Q_n(i\omega)\|_{AP} \|v_0\|_{L^1(\rr)}
\end{equation}
and in view of Lemma \ref{imp} we have
$$\int_{\rr}|(i\omega c_{2n}(i\omega))^\vee(x)|dx\leq C\Big(M,m, \tan \big(\sum _{k=1}^n \arctanh |d_n|\big)\Big).$$
Thus by \eqref{indentity.bv} and \eqref{220a}  the proof of the theorem in  this case is finished. 

\textbf{Case 2. Computing $c_{2n}$ when  ${\rm{supp}}\, v_0\subset I_1 $. } Here all the terms $t_k$, $k=2,\ldots,n-1$ vanish and 
$t_{1,2}=-b_2t_{1,1}$. Hence
\begin{align*}
c_{2n}(i\omega)&=(\det D_n(i\omega))^{-1}
\left|\begin{array}{cccccccc}{\bf a}_1 & B_1 & 0 & 0 & 0 & 0 & 0 & t_{1} \\0 & A_2 & B_2 & 0 & 0 & 0 & 0 & 0 \\0 & 0 & A_3 & B_3 & 0 & 0 & 0 & 0 \\- & - & - & - & - & - & - & - \\0 & 0 & 0 & 0 & A_{n-3} & B_{n-3} & 0 & 0 \\0 & 0 & 0 & 0 & 0 & A_{n-2} & B_{n-2} & 0 \\0 & 0 & 0 & 0 & 0 & 0 & A_{n-1} & 0\end{array}\right|.
\end{align*}
Developing the above determinant over blocks of two lines  we obtain that
\begin{align*}
c_{2n}(i\omega)&=\frac{\det A_2\cdots\det A_{n-1}}{\det D_n(i\omega)} 
\left|\begin{array}{cc}e^{i\omega b_{1}x_{1}} & t_{1,1} \\b_2e^{i\omega b_{1}x_{1}} & t_{1,2}\end{array}\right|
=\frac{\det A_1 \cdots\det A_{n-1}}{\det D_n(i\omega)} e^{i\omega b_1 x_1}t_{1,1}(i\omega).
%&=\frac{t_{11}e^{i\omega b_1 x_1}}{\det D_n(i\omega)}\prod_{k=2}^n(-2b_k).
\end{align*}
Using  estimate \eqref{t.1} on $t_{1,1}$ we get 
%
%
%\eqref{t1} we have
%$$i\omega t_{11}(i\omega)=-\frac{b_1}2 e^{-i\omega b_1x_1} \widehat{v_0 I_1}(-b_1\omega)$$
%and thus
%$$i \omega c_{2n}(i\omega)=\frac{2^{n-2}}{\det D_n(i\omega)}\prod _{k=1}^n (-b_k)\widehat{v_0I_n}(-b_n\omega).$$
%Hence
\begin{equation}\label{caso.2}
\int _{\rr} | {(i\omega c_{2n}(i\omega))}^\vee(x)|dx\leq \left\|\frac{\det A_1 \cdots\det A_{n-1}}{\det D_n(i\omega)} \right\|_{AP} \|v_0\|_{L^1(\rr)}.
\end{equation}
Since  the proof of estimate \eqref{main.ineq} is the same as in Case 3 below (choose  $k=1$ in \eqref{a1}) we will skip it here.

\textbf{Case 3. Computing $c_{2n}$ when  ${\rm{supp}}\, v_0\subset I_k $, $2\leq k \leq n-1$.} Here $t_{k-1,2}=b_{k-1}t_{k-1,1}$, 
$t_{k,2}=-b_{k+1}t_{k,1}$, 
 all the other terms in  vector $T$ vanishing.
Let us now compute $c_{2n}$. It is given by 
\begin{align*}
c_{2n}&(i\omega)(\det D_n(i\omega))=
\left|\begin{array}{cccccccc}
a_1 & B_1 & 0 & 0 & 0 & 0 & 0 & 0 \\
0 & A_2 & B_2 & 0 & 0 & 0 & 0 & 0 \\
0 & 0 & A_3 & B_3 & 0 & 0 & 0 & 0 \\
- & - & - & - & - & - & - & - \\
-&-&-& A_{k-1} & B_{k-1} & -&-& t_{k-1}\\
0 & 0 & 0 & 0 & A_{k} & B_{k} & 0 & t_k \\
- & - & - & - & - & - & - & - \\
0 & 0 & 0 & 0 & 0 & A_{n-2} & B_{n-2} & 0 \\
0 & 0 & 0 & 0 & 0 & 0 & A_{n-1} & 0\end{array}\right|\\
=&
\left|\begin{array}{cccccccc}
a_1 & B_1 & 0 & 0 & 0 & 0 & 0 & 0 \\
0 & A_2 & B_2 & 0 & 0 & 0 & 0 & 0 \\
0 & 0 & A_3 & B_3 & 0 & 0 & 0 & 0 \\
- & - & - & - & - & - & - & - \\
-&-&-& A_{k-1} & B_{k-1} & -&-& t_{k-1}\\
0 & 0 & 0 & 0 & A_{k} & B_{k} & 0 & 0 \\
- & - & - & - & - & - & - & - \\
0 & 0 & 0 & 0 & 0 & A_{n-2} & B_{n-2} & 0 \\
0 & 0 & 0 & 0 & 0 & 0 & A_{n-1} & 0\end{array}\right|+
\left|\begin{array}{cccccccc}
a_1 & B_1 & 0 & 0 & 0 & 0 & 0 & 0 \\
0 & A_2 & B_2 & 0 & 0 & 0 & 0 & 0 \\
0 & 0 & A_3 & B_3 & 0 & 0 & 0 & 0 \\
- & - & - & - & - & - & - & - \\
-&-&-& A_{k-1} & B_{k-1} & -&-& 0\\
0 & 0 & 0 & 0 & A_{k} & B_{k} & 0 & t_k \\
- & - & - & - & - & - & - & - \\
0 & 0 & 0 & 0 & 0 & A_{n-2} & B_{n-2} & 0 \\
0 & 0 & 0 & 0 & 0 & 0 & A_{n-1} & 0\end{array}\right|.
\end{align*}
Developing the determinants over blocks of two lines   we find that
\begin{align*}
c_{2n}(i\omega)(\det D_n(i\omega))&= \det A_k\cdots \det A_{n-1} 
\left|\begin{array}{cccccccc}a_1 & B_1 & 0 & 0 & 0 & 0 & 0 & 0 \\0 & A_2 & B_2 & 0 & 0 & 0 & 0 & 0 \\0 & 0 & A_3 & B_3 & 0 & 0 & 0 & 0 \\- & - & - & - & - & - & - & - \\0 & 0 & 0 & 0 & A_{k-3} & B_{k-3} & 0 & 0 \\0 & 0 & 0 & 0 & 0 & A_{k-2} & B_{k-2} & 0 \\0 & 0 & 0 & 0 & 0 & 0 & A_{k-1} & t_{k-1}\end{array}\right|\\
&+ \det A_{k+1}\cdots \det A_{n-1} 
\left|\begin{array}{cccccccc}a_1 & B_1 & 0 & 0 & 0 & 0 & 0 & 0 \\0 & A_2 & B_2 & 0 & 0 & 0 & 0 & 0 \\0 & 0 & A_3 & B_3 & 0 & 0 & 0 & 0 \\- & - & - & - & - & - & - & - \\0 & 0 & 0 & 0 & A_{k-2} & B_{k-2} & 0 & 0 \\0 & 0 & 0 & 0 & 0 & A_{k-1} & B_{k-1} & 0 \\0 & 0 & 0 & 0 & 0 & 0 & A_{k} & t_{k}\end{array}\right|.
\end{align*}
Since the components of $t_{k-1}$ satisfy $t_{k-1,2}=b_{k-1}t_{k-1,1}$ we can use the same argument as in Case 1 and we obtain that the first determinant equals
$$-t_{k-1,1}e^{-i\omega b_{k}x_{k-1}}\det \tilde D_k(i\omega).$$
The second one could be computed in a similar way by expanding the determinant over the last two lines and using that $t_{k,2}=-b_{k+1}t_{k,1}$ 
\begin{align*}
&
\left|\begin{array}{cccccccc}a_1 & B_1 & 0 & 0 & 0 & 0 & 0 & 0 \\0 & A_2 & B_2 & 0 & 0 & 0 & 0 & 0 \\0 & 0 & A_3 & B_3 & 0 & 0 & 0 & 0 \\- & - & - & - & - & - & - & - \\0 & 0 & 0 & 0 & A_{k-2} & B_{k-2} & 0 & 0 \\0 & 0 & 0 & 0 & 0 & A_{k-1} & B_{k-1} & 0 \\0 & 0 & 0 & 0 & 0 & 0 & A_{k} & t_{k}\end{array}\right|\\[10pt]
&=
-  \left|\begin{array}{cc}e^{i\omega b_kx_k} & t_{k,1} \\b_{k+1}e^{i\omega b_kx_k} & -b_{k+1}t_{k,1}\end{array}\right|\det D_k(i\omega)+
 \left|\begin{array}{cc}e^{-i\omega b_kx_k} & t_{k,1} \\ -b_{k+1}e^{-i\omega b_kx_k} & -b_{k+1}t_{k,1}\end{array}\right|\det \tilde D_k(i\omega)\\[10pt]
 &=2b_{k+1}t_{k,1}\det D_k(i\omega)e^{i\omega b_kx_k}=-\det (A_k)t_{k,1}\det D_k(i\omega)e^{i\omega b_kx_k}.
\end{align*}
This gives us that
\begin{align*}
c_{2n}(i\omega)&=-\frac{ \det A_k\cdots \det A_{n-1}}{\det D_n(i\omega)}\Big(t_{k-1,1}e^{-i\omega b_kx_{k-1}}\det \tilde D_k(i\omega)+
t_{k,1}e^{i\omega b_kx_k}\det D_k(i\omega)\Big)\\
&=-\frac{ \det A_k\cdots \det A_{n-1}\det D_k(i\omega)}{\det D_n(i\omega)}\Big(t_{k-1,1}e^{i\omega b_k(x_k-x_{k-1})}Q_k(i\omega)+
t_{k,1}\Big)e^{i\omega b_kx_k}.
%\\
%&=\prod _{j=k}^{n-1} \frac{-2b_{j+1}e^{-i\omega (b_j-b_{j+1})x_j}}{(b_j+b_{j+1})(1-d_jQ_j(i\omega))}\Big(t_{k-1,1}e^{2i\omega b_kx_k-i\omega b_kx_{k-1}}Q_k(i\omega)+
%t_{k1}e^{i\omega b_kx_k}\Big) \\
%&=
%\prod _{j=k}^{n-1} \frac{(1-d_j)e^{-i\omega (b_j-b_{j+1})x_j}}{(1-d_jQ_j(i\omega))}\Big(t_{k-1,1}e^{2i\omega b_kx_k-i\omega b_kx_{k-1}}Q_k(i\omega)+
%t_{k1}e^{i\omega b_kx_k}\Big)
\end{align*}

%From \eqref{t1} we have
%$$i\omega t_{k,1}(\omega)=-\frac {b_k}2 e^{-i\omega b_kx_k} \widehat{v_0I_k}(-b_k\omega)$$
%and
%$$i\omega t_{k-1,1}(\omega)=\frac {b_k}2 e^{-i\omega b_kx_k} \widehat{v_0I_k}(b_k\omega).$$
Applying estimate \eqref{t.1} for $t_{k,1}$ and $t_{k-1,1}$ we obtain that
%\begin{align*}
%\int _{\rr} | {(i\omega c_{2n}(i\omega))}^\vee(t)|dt& \leq \frac{b_k}2 \left(\Big\|Q_k(i\omega)\prod _{j=k}^{n-1} \frac{(1-d_j)}{(1-d_jQ_j(i\omega))} \Big\|_{AP}+\Big\|\prod _{j=k}^{n-1} \frac{(1-d_j)}{(1-d_jQ_j(i\omega))} \Big\|_{AP}\right) \|v_0\|_{L^1(\rr)}\\
%&\leq \frac {b_k}2(\|Q_k\|_{AP}+1)\Big\|\prod _{j=k}^{n-1} \frac{(1-d_j)}{(1-d_jQ_j(i\omega))} \Big\|_{AP}.
%\end{align*}
\begin{align}\label{a1}
\int _{\rr} & | {(i\omega   c_{2n}(i\omega))}^\vee(t)|dt \\
 \nonumber &\leq C(m) (\|Q_k\|_{AP}+1)\Big\| \frac{ \det A_k\cdots \det A_{n-1}\det D_k(i\omega)}{\det D_n(i\omega)}\Big\|_{AP}\|v_0\|_{L^1(\rr)}.
\end{align}
Using now formula \eqref{Dn} that gives us the explicit expression of $\det D_n$ and $\det D_k$ we find that
\begin{align}\label{a2}
\frac{ \det A_k\cdots \det A_{n-1}\det D_k(i\omega)}{\det D_n(i\omega)}&=\prod _{j=k}^{n-1} \frac{\det A_j}{b_j+b_{j+1}} \frac{e^{-i\omega (b_j-b_{j+1})x_j} } {(1-d_jQ_j(i\omega))}
\\
\nonumber & =(-1)^{n-k}\prod _{j=k}^{n-1} \frac{(1-d_j)}{(1-d_jQ_j(i\omega))}e^{-i\omega (b_j-b_{j+1})x_j}.
\end{align}
Observe now that 
$$\frac{(1-d_j)}{(1-d_jQ_j(i\omega))}=
\left\{
\begin{array}{rr}
\frac{1}{1+d_j} (1+d_jQ_{j+1}(i\omega)e^{2i\omega b_{j+1}(x_{j+1}-x_j)}),& j=1,\ldots,n-2,\\[10pt]
\frac{1}{1+d_j} (1+d_jQ_{j+1}(i\omega)e^{-2i\omega b_{j+1}x_j}),& j=n-1.
\end{array}
\right.
$$
It gives us that for any $1\leq j\leq n-1$ we have
$$\Big\|\frac{(1-d_j)}{(1-d_jQ_j(i\omega))} \Big\|_{AP}\leq \frac {1+|d_j|\|Q_{j+1}\|_{AP} }{1+d_j}\leq \frac{\exp(|d_j|\|Q_{j+1}\|_{AP})}{1+d_j}$$
and then
\begin{align}\label{a3}
\Big\|\prod _{j=k}^{n-1} &\frac{(1-d_j)}{(1-d_jQ_j(i\omega))} \Big\|_{AP}
\leq \exp\Big(\max_{j=k,\ldots, n-1} \|Q_{j+1}\|_{AP}\sum _{j=k}^{n-1} |d_j| \Big).
 \prod _{j=k}^{n-1}\frac 1{1+d_j}\\
\nonumber &\leq \exp\Big(c(m,M)\var (a)\max_{j=k+1,\ldots, n} \|Q_{j}\|_{AP}\Big).
 \prod _{j=k}^{n-1}\frac 1{1+d_j}\\
  \nonumber &\leq \exp\Big(c(m,M)\var (a) \tan (\var(\log (a)) /4) \Big).
 \prod _{j=k}^{n-1}\frac 1{1+d_j}.
\end{align}
The last term also satisfies
\begin{equation}\label{a4}
\prod _{j=k}^{n-1}\frac 1{1+d_j}=\prod _{j=k}^{n-1}\frac {b_j+b_{j+1}}{2b_{j}}=\prod _{j=k}^{n-1}\Big (1+\frac{|b_{j+1}-b_j|}{2b_j}\Big)\leq \exp\Big(C(m,M)\var (a)\Big).
\end{equation}
Putting now toghether estimates \eqref{a1}, \eqref{a2}, \eqref{a3} and \eqref{a4} we obtain that estimate \eqref{main.ineq}
also holds in the case considered here.

\textbf{Case 4. Prove \eqref{main.ineq} when ${\rm{supp}}\, v_0\subset I_k $, $2\leq k\leq n-1$ and $k\leq j\leq n-1$. } 
The previous cases prove \eqref{main.ineq} for $j=n$. Let us now prove that it holds for any $k\leq j\leq n-1$. We point out that once estimate \eqref{main.ineq} will be proved then it also holds for $1\leq j\leq k-1$. 

We now use that 
\begin{equation}\label{a}
A_{n-1}\left(\begin{array}{c}c_{2n-3} \\c_{2n-2}\end{array}\right)+{\bf b}_{n-1}c_{2n}=
 \left(\begin{array}{c}0\\ 0
 \end{array}\right)
\end{equation}
and for $j\leq n-2$,
\begin{equation}\label{b}
A_j \left(\begin{array}{c}c_{2j-1} \\c_{2j}\end{array}\right)+B_j \left(\begin{array}{c}c_{2j+1} \\c_{2j+2}\end{array}\right)=
 \left(\begin{array}{c}0\\ 0
 \end{array}\right).
\end{equation}
 From identity \eqref{a} and the results of Case 3, we obtain that \eqref{main.ineq} holds for $j=n-1$:
 \begin{align*}
 \int _{\rr} \Big(|{(i\omega c_{2n-2}(i\omega))}^\vee(x)|+&
|{(i\omega c_{2n-3}(i\omega))}^\vee(x)|\Big)dx\leq C(m,M)\int _{\rr} |{(i\omega c_{2n}(i\omega))}^\vee(x)|dx\\
&\leq C(m,M,\var(a))\int _{\rr}|v_0(x)|dx. 
 \end{align*}
Applying  identity \eqref{b} we obtain  
\begin{align*}
\left(\begin{array}{c}c_{2j-1} \\c_{2j}\end{array}\right)=&-A_j^{-1}B_j \left(\begin{array}{c}c_{2j+1} \\c_{2j+2}\end{array}\right)\\
=&
\frac {1}{2b_{j+1}}\left(\begin{array}{cc}(b_{j+1}+b_j)e^{i\omega(b_{j+1}-b_j)x_j} & (b_{j+1}-b_j)e^{-i\omega(b_{j+1}+b_j)x_j} \\(b_{j+1}-b_j)e^{i\omega(b_{j+1}+b_j)x_j} & (b_{j+1}+b_j)e^{-i\omega(b_{j+1}-b_j)x_j}\end{array}\right)
\left(\begin{array}{c}c_{2j+1} \\c_{2j+2}\end{array}\right).
\end{align*} 
It implies that for $j\leq n-2$ we have
\begin{align*}
\int _{\rr} |{(i\omega c_{2j-1}(i\omega))}^\vee(t)|&+
|{(i\omega c_{2j}(i\omega))}^\vee(t)|dt \\
\leq &\frac{|b_{j+1}-b_j|+b_j+b_{j+1}}{2b_{j+1}}\int _{\rr}|{(i\omega c_{2j+1}(i\omega))}^\vee(t)|+
|{(i\omega c_{2j+2}(i\omega))}^\vee(t)|dt\\
\leq & \left(1+\frac{|b_{j+1}-b_j|}{b_{j+1}}\right)\int _{\rr}|{(i\omega c_{2j+1}(i\omega))}^\vee(t)|+
|{(i\omega c_{2j+2}(i\omega))}^\vee(t)|dt.
\end{align*}
Using the same argument as in \eqref{a4} we obtain that 
for any $k\leq j\leq n-1$
\begin{equation}\label{}
\int _{\rr} |{(i\omega c_{2j-1}(i\omega))}^\vee(t)|+
|{(i\omega c_{2j}(i\omega))}^\vee(t)|dt \leq C(m,M,\var(a))\int _{\rr}|v_0|. 
\end{equation}
The proof of Theorem \ref{disp.wave} is now complete.
\end{proof}

\begin{proof}[Proof of Theorem \ref{no-disp}]
In this section we describe the manner in which the  coefficient $a$ satisfying the conditions in Theorem \ref{no-disp} can be constructed.
Without restricting the generality, we will construct a coefficient $a$ with $1/2\leq a\leq 2$. 

Let fix $d=\tanh((\log 2)/2)$. Since   $\alpha\geq \pi/2$, by Lemma \ref{negative},  for any $N>0$ there exists a sequence $\{\tilde d_k\}_{k=1}^{n-1}$ such that 
$$\sum _{k=1}^{n-1}\arctanh \tilde d_k=\alpha, \quad  0\leq \ \tilde d_k\leq d, \ \forall\ k=1,\ldots, n-1,
$$
and for any sequence of rationally independent  numbers $\{c_k\}_{k=1}^n$ and any  $\eps_k\in \{\pm 1\}$, $k=1,\ldots, n-1$, 
the  sequence $\{Q_k\}_{k=1}^n$ associated to $\{c_k\}_{k=1}^{n-1}$ and $\{\eps_k \tilde d_k\}_{k=1}^{n-1}$
 satisfies
$\|Q_n(i \omega)\|_{AP}\geq 2N$.

Let us set  $d_k=\eps_k \tilde d_k=\eps _k |d_k|$ where $\eps_k\in \{\pm 1\}$ will be chosen later.
We now  choose a sequence $\{b_k\}_{k=1}^n$ such that
$$d_k=\eps _k|d_k|=\frac{b_{k}-b_{k+1}}{b_{k}+b_{k+1}}<0,\quad k=1,\ldots, n-1.$$
This is possible since we can define recursively 
$$\frac{b_k}{b_{k+1}}=\exp(2\arctanh d_k), k= 1,\ldots, n-1, \quad b_1=1.$$

Let    $a$  be the piecewise constant function defined by
$$a(x)=b_k^{-2}, \quad x\in (x_{k-1},x_k),\quad k=1,\ldots, n,$$
where $x_0=-\infty$, $x_n=\infty$  and the points $\{x_k\}_{k=1}^{n-1}$ are chosen such that the numbers
$$b_k(x_k-x_{k-1}),\quad k=2,\ldots, n-1$$
are linearly independent over $\QQ$. 
Using \eqref{indentity.bv} we get that
$$\var(\log (a))=4 \sum _{k=1}^{n-1}\arctanh |d_k|=4 \alpha.$$
We now show how to  choose the sequence $\{\eps_k\}_{k=1}^{n-1}$ such that $a\in [1/2,2]$.
Observe that for $x\in (x_{k-1},x_k)$, the coefficient $a$ is given by
$$a(x)=b_k^{-2}=\exp \Big(2\sum _{j=1}^{k-1} \arctanh d_k\Big) =
\exp \Big(2\sum _{j=1}^{k-1} \eps_k \arctanh |d_k|\Big). $$
Since  $0\leq \arctanh |d_k|\leq (\log 2)/2$ we always can choose $\eps_k\in \{\pm\}$ such that 
$$\Big|\sum _{j=1}^{k-1} \eps_j \arctanh |d_j|\Big|\leq \frac{\log 2}2, \quad \forall \  k=2,\ldots, n-1.$$
Indeed, for $k=2$ choose $\eps_1=1$. Assume that the above inequality is true for some $k$. If $\sum _{j=1}^{k-1} \eps_j \arctanh |d_j|$ belongs to $[0,\log(2)/2]$ then choose $\eps_k=-1$, otherwise $\eps_k=1$. It follows that the above inequality will also hold for $k+1$. With this choice of $\{\eps_k\}_{k=1}^{n-1}$
we obtain that  coefficient $a$ satisfies $a\in [1/2,2]$.

%
%
%We  choose an increasing sequence $\{b_k\}_{k=1}^n$ such that all the reflection coefficients $d_k$ are negative:
%$$d_k=\frac{b_{k-1}-b_k}{b_{k-1}+b_k}<0,\quad k=2,\ldots, n.$$
%Finally by choosing $b_1=1$ we obtain that
%\begin{align*}
%\var(\log (a))=-2\sum _{k=1}^{n} \log (b_{k-1}/b_k)=4\alpha\geq 2\pi.
%\end{align*}

Denoting $m=1/2$, in view of estimate \eqref{est.tilde} we have
\begin{align*}
\frac{\displaystyle\sup_{x\in \rr}\int_{\rr} |v(t,x)|dt}{\|v_0\|_{L^1(\rr)}}\geq 
\frac{\displaystyle\int _{\rr} |i\omega c_{2n}(i\omega) ^\vee (t)|dt}{\|v_0\|_{L^1(\rr)}}-m^{-2}\gtrsim 
\frac{\displaystyle \|(Q_n(i\omega) \widehat v_0 )^\vee)\|_{L^1(\rr)}}{\|v_0\|_{L^1(\rr)}}-m^{-2}.
\end{align*} 
Let us now a sequence $v_{0k}$ such that  $$\frac{\displaystyle \|(Q_n(i\omega) \widehat v_{0k} )^\vee)\|_{L^1(\rr)}}{\|v_{0k}\|_{L^1(\rr)}}\rightarrow \|Q_n\|_{AP}.$$
This is always possible since $Q_n(i\omega)$ is an $L^1(\rr)$- Fourier multiplier and its norm is given by $\|Q_n\|_{AP}$, see \cite[Th. 2.5.8, p. 141]{MR2445437}. Thus
we get for any $N$ large enough that 
$$
\displaystyle \sup_{v_0\in L^1(\rr)}\frac{\displaystyle\sup_{x\in \rr}\int_{\rr} |v(t,x)|dt}{\|v_0\|_{L^1(\rr)}}\gtrsim
\|Q_n\|_{AP}-m^{-2}\geq 2N-4\geq N.$$
For the function $a$ chosen above  we obtain that the left hand side of the above inequality can be arbitrarily large. The proof if now finished.
\end{proof}

\begin{proof}[Proof of Theorem \ref{disp.sch}]
We use formula \eqref{res} for the resolvent to find that for $x\in (x_{k-1},x_k)$  solution $u$ of system \eqref{sch} is given by
\begin{align*}
u(t,x)=&\frac 1{i\pi}\int _{\rr}e^{-it\omega^2} \omega c_{2k-1}(i\omega)e^{i\omega b_kx}d\omega+\frac 1{i\pi}\int _{\rr}e^{-it\omega^2} \omega c_{2k}(i\omega)e^{-i\omega b_kx}d\omega\\
&+\frac{b_k}{2i \pi}\int _{\rr} e^{-it\omega^2} \omega \int _{I_k}u_0(y)e^{-i\omega_k|x-y|}dyd\omega.
\end{align*}
It follows that
$$|u(t,x)|\leq \frac {C}{\sqrt t}\int _{\rr} \big ( |(i\omega c_{2k-1}(i\omega))^\vee (x) |+  |(i\omega c_{2k}(i\omega))^\vee (x) |\big)dx+\frac {C}{\sqrt t}\|u_0\|_{L^1(\rr)}.$$
Using now estimate \eqref{main.ineq} that has already been proved in the proof of Theorem \ref{disp.wave} we obtain the desired estimate and the proof is finished.
\end{proof}

\section{Proof of the two technical lemmas}\label{proof.lemmas}
In this section we prove Lemma \ref{imp} and Lemma \ref{negative} using a fine analysis of the sequence $\{Q_n\}_{n\geq 1}$ defined by \eqref{def.q} by means of multi-variable series.

\subsection{The functions $\a (q)$, $\b (d)$, $\c (t)$,
$E(d_1,\ldots,d_n;q_1,\ldots,q_{n-1})$ and
$R(t_1,\ldots,t_n;q_1,\ldots,q_{n-1})$.}

We introduce the  notations that will be used in this section.
We denote by $\a (q)$ and $\b (\sigma)$ the mappings:
$$\a(q): z\mapsto qz, \quad \b(\sigma): z\mapsto\frac{z+\sigma}{1+\sigma z}.$$
Note that $q\mapsto\a (q)$ is a multiplicative mapping, i.e. $\a (q)\a
(q')=\a (qq')$, and $\beta$ satisfies
%For $\b$ we note that 
%$$\left(\begin{matrix}
%1&d\\ d&1 \end{matrix} 
%\right)\left( \begin{matrix}
%1&d'\\ d'&1 \end{matrix} \right) =\left( \begin{matrix}
%1+dd'&d+d'\\ d+d'&1+dd' \end{matrix} \right)$$ so $\b (d)\b (d')$
%is given by
%$$z\mapsto\frac{(1+dd')z+d+d'}{(d+d')z+1+dd'}=
%\frac{z+\frac{d+d'}{1+dd'}}{\frac{d+d'}{1+dd'}z+1}$$
 $$\b (\sigma)\b (\sigma')=\b \Big(\frac{\sigma+\sigma'}{1+\sigma \sigma'}\Big).$$
  From the formula 
  $$\tanh
(x+y)=\frac{\tanh x+\tanh y}{1+\tanh x\tanh y}$$ we deduce that 
mapping $t\mapsto\c (t):=\b(\tanh t)$ has the property 
$$\c (a)\c (b)=\c (a+b), \quad a,b\in \rr.$$

%\textcolor{red}{Then $q\mapsto\beta (q)$ is a morphism
%from $\CC^\times$ and $PGL(2,\CC )$. }

%====================

%nicu: explica chestia asta cu PGL ca nu o sa fie revista de algebra sa stie lum%ea

%====================

If $d=(d_1,\ldots,d_n)$, with $0\leq d_i<1$ and $q=(q_1,\ldots,q_{n-1})$
is a multivariable then we define the multivariable series
$$E(d;q):=\left(\b (d_n)\,\a (q_{n-1})\,\b (d_{n-1})\cdots\a (q_1)\b
(d_1)\right) (0).$$
With the convention that for 
$j=(j_1,\ldots,j_{n-1})$ we write $q^j:=q_1^{j_1}\cdots
q_{n-1}^{j_{n-1}}$. We introduce  the norm $||\cdot||$ of a multivariable series  by 
$$\Big\|\sum_{j\in\NN^{n-1}}a_jq^j\Big\|=\sum_{j\in \NN^{n-1}}|a_j|.$$ 
It is easy to see that for any two series $\|ab\|\leq \|a\| \|b\|$.

%=====================

%nicu: e nevoie sa presupun ca $d_i$ sunt in $(0,1)$ sau merge si pt $[0,1)$?

%=====================

\begin{lemma}\label{ineg.q.E}
Let be $\{Q_k\}_{k\geq 1}$  defined by \eqref{def.q}. Then for any $n\geq 1$
\begin{equation}\label{QleqE}
||Q_{n+1}||_{AP}\leq
||E(|d_1|,\ldots,|d_n|;q_1,\ldots, q_{n-1})||,
\end{equation}
with equality when $c_1,\ldots,c_{n-1}$
are linearly independent over $\QQ$. 
\end{lemma}

\begin{proof}
We write $d_k=\e_k|d_k|$ with $\e_k=\pm 1$. The sequence
$\{Q_k\}_{k\geq 1}$ can be written as 
\begin{align*}
Q_k(i\omega)&=-\e_{k-1}e^{i\omega
 c_{k-1}}\frac{-\e_{k-1}Q_{k-1}(i\omega
)+|d_{k-1}|}{1-\e_{k-1}|d_{k-1}|Q_{k-1}(i\omega
 )}\\
&=\left(\a (-\e_{k-1}e^{i\omega
 c_{k-1}})\,\b (|d_{k-1}|)\,\a
(-\e_{k-1})\right)(Q_{k-1}(i\omega
 )), k\geq 2.
\end{align*}
It follows that
\begin{multline*}
Q_{n+1}(i\omega )
=\Big(\a (-\e_ne^{i\omega c_n})\,\b (|d_n|)\,\a (\e_{n-1}\e_ne^{i\omega
c_{n-1}})\,\b (|d_{n-1}|)\cdots\a (\e_1\e_2e^{i\omega c_1})\,\b
(|d_1|)\Big) (0).
\end{multline*}
With the above notations it follows that
$$Q_{n+1}(i\omega )=-\e_ne^{i\omega c_n}\,
E(|d_1|,\ldots,|d_n|;\e_1\e_2e^{i\omega c_1},\ldots,\e_{n-1}\e_ne^{i\omega
c_{n-1}}).$$
 Since $|-\e_n|=1$ and $|\e_k\e_{k+1}|=1$ for $1\leq k\leq n-1$ we have 
\begin{equation*}
||Q_{n+1}||_{AP}\leq
||E(|d_1|,\ldots,|d_n|; q_1,\ldots, q_{n-1})||,
\end{equation*}
with equality when $c_1,\ldots,c_{n-1}$
are linearly independent over $\QQ$. 
\end{proof}

%\medskip
%If $A=\left(\begin{smallmatrix}0&1\\ 1&0\end{smallmatrix}\right)$ then
%the mapping $t\mapsto\exp (tA)$ is a morphism from $\RR$ to $GL(2,\RR
%)$. Since $A^2=I$ we have
%\begin{align*}
%\exp (tA)&=\sum_{n\geq 0}\frac{t^n}{n!}A^n=\sum_{n\text{
%even}}\frac{t^n}{n!}I+\sum_{n\text{ odd}}\frac{t^n}{n!}A\\
%&=\cosh tI+\sinh
%tA=\begin{pmatrix}\cosh t&\sinh t\\ \sinh t&\cosh t\end{pmatrix}.
%\end{align*}
%
%
%Hence if we define $\c (t)$ by 
%$$\c (t): z\mapsto\frac{\cosh t\,z+\sinh t}{\sinh
%t\,z+\cosh t}=\frac{z+\tanh t}{1+z\tanh t}$$
% then the mapping
%$t\mapsto\c (t)$ is a morphism from $\RR$ to $PGL(2,\RR )$. 

In the following we will estimate in a clever way the norm of $E(|d_1|,\ldots,|d_n|; q_1,\ldots, q_{n-1})$.
%Let us now define $\gamma(t)=\beta(\tanh (t))$, $t\in  \rr$. The main advantage of the new map $\gamma$ is the following $\log$-additive property:
%$$\gamma(a+b)=\gamma(a)\gamma(b),\quad \forall \, a,b\in \rr.$$
%
For any  $t=(t_1,\ldots,t_n)$  we define the series
$R(t;q)$ in  the multivariable $q=(q_1,\ldots,q_{n-1})$ by
\begin{align}\label{formula.R}
R(t;q)&=\left(\c (t_n)\,\a (q_{n-1})\,\c (t_{n-1})\cdots\a (q_1)\c
(t_1)\right) (0)\\
\nonumber&=E(\tanh t_1,\ldots, \tanh t_n; q).
\end{align}
When there is no risk of confusion we will write only $R(t)$.

The fact that $\tanh$ gives a bijection between $[0,\infty )$ and
$[0,1)$ allows us to find estimates for
$||E(d_1,...,d_n;q_1,...,q_{n-1})||$ in terms of estimates for
$||R(t_1,...,t_n;q_1,...,q_{n-1})||$. Dealing with $R$ rather than
$E$ is more advantageous because it allows us to use the formula
$\gamma (a)\gamma (b)=\gamma (a+b)$. 
%Since $\c (t)=\b (\tanh t)$ we have 
%$$R(t_1,\ldots,t_n)=E(\tanh
%t_1,\ldots,\tanh t_n).$$

\subsection{Estimates on partitions.}

For any positive number  $x$ we consider the set of all partitions of $x$:
$$A_x=\Big\{ t=(t_1,\ldots,t_n)\mid
t_j>0,\,\sum_{j=1}^nt_j=x\Big\}.$$ 
%The elements of $A_x$ are called
%partitions of $x$. 
If $t=(t_1,\ldots,t_n)\in A_x$  the sequence of the partial sums
$x_k=\sum_{j=1}^kt_j$, $k=0,\ldots, n$ is a partition of the interval
$[0,x]$, i.e. we have $0=x_0<x_1<\ldots <x_n=x$. We get a one-to-one
correspondence between $A_x$ and the partitions of $[0,x]$. 

If $t=(t_1,\ldots,t_n)$ and $s=(s_1,\ldots,s_m)$ are two elements in
$A_x$ we say that $s$ is
finer than $t$ and we write $t\prec s$ if the sequence of partial sums
of $s$ contains the sequence of partial sums of $t$, i.e. there is a
sequence $1\leq k_1<\ldots <k_n=m$ such that
$$\sum_{j=1}^lt_j=\sum_{j=1}^{k_l}s_j, \quad \forall\, 1\leq l\leq n.$$
%i.e. such that 
%$$t_l=\sum_{j=k_{l-1}+1}^{k_l}s_j, \quad \forall 1\leq l\leq
%n.$$ 
It follows that $(A_x,\prec )$ is a directed set.

\begin{lemma} \label{expansion.4}
Let be $0\leq d_j\leq 1$, $1\leq j\leq n$ and  $P=\prod_{j=1}^n (1+d_j)$,
$p=\prod _{j=1}^n (1-d_j)$. Then the map
$$\b (d_n)\a (q_{n-1})\cdots \b (d_1)\a (q_0)$$ is
given by the map $z\mapsto (az+b)/(cz+d)$, where $a,b,c,d\in\RR
[q_0,\ldots,q_{n-1}]$ have non-negative coefficients,
$$||a||=||c||=\frac 12(P+p), ||b||=||d||=\frac 12(P-p),$$  the
coefficient of $q_0\cdots q_{n-1}$ in polynomial $a$  
 and the constant term in polynomial $d$ is $1$.
%
%Moreover when $n=1$, 
%$$\begin{pmatrix}a&b
%\\ c&d
%\end{pmatrix}
%=\begin{pmatrix}
%q_0&d_1\\
% d_1 q_0&1.
% \end{pmatrix}$$
\end{lemma}
\begin{proof} Note that 
$$\frac 12(P+p)=\sum_{k\text{ even}}\sum_{1\leq
i_1<\ldots<i_k\leq n}d_{i_1}\cdots d_{i_k}$$ and
 $$\frac
12(P-p)=\sum_{k\text{ odd}}\sum_{1\leq
i_1<\ldots<i_k\leq n}d_{i_1}\cdots d_{i_k}.$$

Let $M(q_0,\ldots,q_{n-1})$ be the matrix
$$M(q_0,\ldots,q_{n-1})=\begin{pmatrix}a&b\\ c&d\end{pmatrix}=
\begin{pmatrix}1&d_n\\ d_n&1\end{pmatrix}
\begin{pmatrix}q_{n-1}&0\\ 0&1\end{pmatrix} 
\cdots\begin{pmatrix}1&d_1\\ d_1&1\end{pmatrix}
\begin{pmatrix}q_0&0\\ 0&1\end{pmatrix}.$$
Thus the map $\b (d_n)\a (q_{n-1})\cdots\b (d_1)\a (q_0)$ is
given by $$z\mapsto\frac{az+b}{cz+d}.$$
Since $d_j$, $j\geq 1$ are nonnegative, $a,b,c,d\in R[q_0,\ldots,q_{n-1}]$ are
polynomials with non-negative coefficients. Therefore
$||a||=a(1,\ldots,1)$ and similarly for $b,c,d$.

Denoting  $$A=\left(\begin{matrix}0&1\\ 1&0\end{matrix}\right),
B_0=\left(\begin{matrix}0&0\\ 0&1\end{matrix}\right),
B_1=\left(\begin{matrix}1&0\\ 0&0\end{matrix}\right)$$
 we have that
 \begin{equation}\label{M}
M(q_0,\ldots,q_{n-1})=(I+d_nA)(B_0+q_{n-1}B_1)\cdots (I+d_1A)(B_0+q_0B_1).
\end{equation}
Since $B_0+B_1=I$ and $A^2=I$, it follows that
\begin{align*}
M(1,\ldots,1)&=(I+d_nA)\cdots (I+d_1A)=\sum_{k=0}^n\sum_{1\leq
i_1<\ldots<i_k\leq n}d_{i_1}\cdots d_{i_k}A^k\\
 &=\left(\sum_{k\text{ even}}\sum_{1\leq
i_1<\ldots<i_k\leq n}d_{i_1}\cdots d_{i_k}\right)
I+\left(\sum_{k\text{ odd}}\sum_{1\leq
i_1<\ldots<i_k\leq n}d_{i_1}\cdots d_{i_k}\right)A\\
&=\frac 12(P+p)I+\frac
12(P-p)A=\begin{pmatrix}\frac 12(P+p)&\frac 12(P-p)\\[10pt]
 \frac 12(P-p)&\frac
12(P+p)\end{pmatrix}.
\end{align*}
It implies  that $$a(1,\ldots,1)=c(1,\ldots,1)=\frac 12(P+p)$$ and
$$b(1,\ldots,1)=d(1,\ldots,1)=\frac 12(P-p).$$
Then the first part of Lemma \ref{expansion.4} is proved. 
 
 Let us now compute the coefficient of $q_0\cdots q_{n-1}$ in
 polynomial $a$ and the constant term in polynomial $d$.
 In view of \eqref{M} 
we have
$$M(q_0,\ldots,q_{n-1})=\sum_{(k_0,\ldots,k_{n-1})\in\{ 0,1\}^n}q_0^{k_0}\cdots
q_{n-1}^{k_{n-1}}M_{k_0,\ldots,k_{n-1}},$$
where
$$M_{k_0,\ldots,k_{n-1}}=(I+d_nA)B_{k_{n-1}}\cdots (I+d_1A)B_{k_0}.$$
To determine the coefficient of $q_0\cdots q_{n-1}$ and the
constant term of $d$ one has to calculate $M_{1,\ldots,1}$ and
$M_{0,\ldots,0}$, respectively. 

Note that $B_0^2=B_0$, $B_1^2=B_1$ and $B_0AB_0=B_1AB_1=0$ so
$B_0(I+d_jA)B_0=B_0$ and
$B_1(I+d_jA)B_1=B_1$. 
%Inductively one obtains
%$$B_0(I+d_{n-1}A)B_0\cdots B_0(I+d_1A)B_0=B_0$$
%and
%$$B_1(I+d_{n-1}A)B_1\cdots B_1(I+d_1A)B_1=B_1.$$
Hence
$$M_{0.\ldots,0}=(I+d_nA)B_0\cdots
(I+d_1A)B_0=(I+d_nA)B_0=\begin{pmatrix}0&d_n\\ 0&1\end{pmatrix}$$
so the constant term of $d$ is one. Also 
$$M_{1.\ldots,1}=(I+d_nA)B_1\cdots
(I+d_1A)B_1=(I+d_nA)B_1=\begin{pmatrix}1&0\\ d_n&0\end{pmatrix}$$
so the coefficient of $q_1\cdots q_{n-1}$ in $a$ is one.
%
%The case when $n=1$ is immediate and the proof is finished.
 \end{proof} 

\begin{corollary} \label{cor.5}
Let $(t_1,\ldots, t_n)\in A_x$. The map
$$\c
(t_n)\a (q_{n-1})\cdots\c (t_1)\a (q_0)$$ is given by
$z\mapsto (az+b)/(cz+d)$, where $a,b,c,d\in\RR
[q_0,\ldots,q_{n-1}]$ have non-negative coefficients,
$$||a||=||d||< \cosh x,\,||b||=||c||< \sinh x,$$
the coefficient of $q_0\cdots q_{n-1}$ in $a$ and the constant
term of $d$ is $1$.
%
% Moreover if $n=1$ then $a=q_0$, $b=\tanh(t_1)$, $c=\tan(t_1) q_0$ and $d=1$.
\end{corollary}
\begin{proof} Since $\c (t)=\b (\tanh t)$ we can use Lemma \ref{expansion.4} with $d_j=\tanh
t_j$. Using the notations of Lemma \ref{expansion.4}, we must prove that 
\begin{equation}\label{est.Pp}
\frac{P+p}2<\cosh x \quad\text{and} \quad \frac {P-p}2<\sinh x.
\end{equation}
Since  $$1+d_j=1+\tanh t_j=\frac {e^{t_j}}{\cosh t_j} ,\quad 1-d_j=1-\tanh
t_j=\frac {e^{-t_j}}{\cosh t_j}$$
we get
$$P=e^x\left(\prod_j\cosh t_j\right)^{-1}<e^x, \,p=e^{-x}\left(\prod_j\cosh t_j\right)^{-1}<e^{-x}.$$ 
Thus \eqref{est.Pp} immediately follows. 
%The last sentence also follows from Lemma \ref{expansion.4}.
%$$\frac{P+p}2=\cosh x\left(\prod_j\cosh t_j\right)^{-1}<\cosh x$$ and 
%$$\frac{P-p}2=\sinh x\left(\prod_j\cosh t_j\right)^{-1}<\sinh x.$$ 
\end{proof} 

\begin{lemma} \label{R-creste}
For any two partitions
$s,t\in A_x$ with $t\prec s$ the following holds for any positive integer $r$
$$||R(t)^r||\leq||R(s)^r||.$$ 
\end{lemma}
\begin{proof}
 Let $t=(t_1,\ldots,t_n)\prec s=(s_1,\ldots,s_m)$ be two
  elements in $A_x$. 
 In order to prove this result we relate the two series $R(t)$ and $R(s)$. Observe that $R(t)$ and $R(s)$ depend 
on   $n-1$  and $m-1$ variables respectively,  $m\geq n$. 
  Since $t\prec s$
there are $0=k_0<k_1<\cdots <k_n=m$ such that
$$t_l=\sum_{j=k_{l-1}+1}^{k_l}s_j,\quad \forall 1\leq l\leq n.$$ 
Using that  $\gamma(a+b)=\gamma(a)\gamma(b)$ it follows
that $$\c (t_l)=\prod_{j=k_{l-1}+1}^{k_l}\c (s_j).$$
 Since $\a (1)=1$ we
also have $$\c (t_l)=\c (s_{k_{l-1}+1})\a (1)\c (s_{k_{l-1}+2})\a
(1)\cdots\a (1)\c (s_{k_l}).$$ 
Replacing $\c (t_l)$ by the above
formula  in the definition of $R(t)$, \eqref{formula.R},  one gets
\begin{align*}
R(t;q_1,\ldots,q_{n-1})&=\left(\c (t_n)\a (q_{n-1})\c (t_{n-1})\cdots\a
(q_1)\c (t_1)\right) (0)\\
&=R(s;1,\ldots,1,q_1,1,\ldots,1,q_2,\ldots,q_{n-1},1\ldots,1),
\end{align*}
where the blocks  of one above have
lengths  $k_n-k_{n-1}-1,\ldots, k_2-k_1-1, k_1-k_0-1$ .
Thus $||R(t)||\leq ||R(s)||$.
For $r\geq 2$ the argument is similar since 
$$R(t; q_1,\ldots,q_{n-1})^r=R(s;1,\ldots,1,q_1,\ldots,q_{n-1},1\ldots,1)^r.$$
The proof is finished.
 \end{proof}

\subsection{Upper bounds for $R(t)$}
We now obtain some properties of the multivariable series $R(t)$ introduced above.
For any integer $r\geq 0$ we define the function $f_r:(0,\infty )\to (0,\infty
]$ by 
%$$f_r(x)=
%\left\{
%\begin{array}{ll}
%\displaystyle \sup_{t\in A_x}||R(t)^r||,& r\geq 1,\\[10pt]
%1, & r=0.
%\end{array}
%\right.
%$$
$$f_r(x)=\sup_{t\in A_x}||R(t)^r||.$$
In particular, $f_0\equiv 1$.
We note that 
$f_{r_1+r_2}(x)\leq f_{r_1}(x)f_{r_2}(x).$ In particular for any integer $r\geq 1$,
$f_r(x)\leq f_1(x)^r.$

The first estimate for $f_1$ is given in the following lemma.
\begin{lemma}\label{first.estimate}
For any $x\in (0,\log (2+\sqrt 3))$ the following holds
\begin{equation}\label{est.f.1}
f_1(x)\leq \frac{\sinh x}{2-\cosh x}.
\end{equation}
\end{lemma}

\begin{proof}
Let us choose $t=(t_1,\ldots, t_n)\in A_x$ and denote  $q=(q_1,\ldots, q_{n-1})$. We will show that
 $$\|R(t;q)\|\leq \frac{\sinh x}{2-\cosh x}. $$
 By definition 
 $$R(t;q)=\Big(\c (t_n)\a (q_{n-1})\c (t_{n-1})\cdots\a(q_1)\c (t_1)\Big) (0).$$
 From Corollary \ref{cor.5} it follows that
 $$\left(\c (t_n)\a (q_{n-1})\c (t_{n-1})\cdots\a
(q_1)\c (t_1) \alpha(q_0) \right) (z)=\frac{az+b}{cz+d}.$$
Since $\a (q_0)(0)=0$ when we take $z=0$ in the above equation we get
$$R(t;q)=\frac bd.$$
Using  Corollary \ref{cor.5}, we have that
$$\| b\|<\sinh x\ \text{and}\ \|d-1\|=\|d\|-1<\cosh x-1.$$
Then, for any $x$ satisfying $\cosh x<2$ the following holds
$$\|R(t;q)\|\leq \frac{\|b\|}{1-\|1-d\|}\leq \frac{\sinh
x}{2-\cosh x}.$$ 
(Note that $a,b,c,d\in\RR [q_0,q_1,\ldots,q_{n-1}]$ but the variable
$q_0$ is superfluous in $\frac bd=R(t;q)$. It does not appear effectively.)

Since the partition $t\in A_x$ has been arbitrarily chosen we obtain
that estimate \eqref{est.f.1} holds for all $x<{\rm arccosh}(2)
=\log(2+\sqrt 3)$.
\end{proof}

\begin{lemma}\label{f.crescatoare}
 For any positive integer $r$ the function $f_r$ is  increasing.
\end{lemma}
\begin{proof} Let us observe that for any $d\in [0,1)$ we have
 $\b (d)\a (-1)\b (d)\a
(-1)=1$. Thus, for any nonnegative $t$,  $\gamma (t)$ satisfies $\c
(t)\a (-1)\c (t)\a (-1)=1$. 

Let us now choose  $0<x<y$,  where $y=x+z$ with $z>0$. Let
$t=(t_1,\ldots,t_n)\in A_x$. Then $s:=(t_1,\ldots,t_n, z/2, z/2)\in
A_y$. Since $\c ( z/2)\a (-1)\c ( z/2)\a (-1)=1$ we have:
\begin{align*}
R(t;q_1,\ldots,q_{n-1})&=\left(\c (\frac z2)\a (-1)\c (\frac z2)\a (-1)\c
(t_n)\a (q_{n-1})\cdots\a (q_1)\c(q_1)\right) (0)\\
&=R(s;(q_1,\ldots,q_{n-1},-1,-1))
\end{align*}
Hence for any integer $r\geq 1$,
$$R(t;q_1,\ldots,q_{n-1})^r=R(s;q_1,\ldots,q_{n-1},-1,-1)^r,$$ which
implies that $||R(t)^r||\leq ||R(s)^r||$. This implies that  for any $t\in A_x$
there is some $s\in A_y$ with $||R(t)^r||\leq ||R(s)^r||$. It follows
that $f_r(x)\leq f_r(y)$ so $f_r$ is an increasing function
%$$f_r(x)=\sup_{t\in A_x}||R(t)^r||\leq\sup_{s\in
%A_y}||R(s)^r||=f_r(y)$$

 \end{proof}

We denote $I=\{ x\in (0,\infty )\mid f_1(x)<\infty\}$. In  view of Lemma \ref{first.estimate} and 
Lemma \ref{f.crescatoare}, the set  $I$ is an interval that includes $(0,\log(2+\sqrt 3))$.
Moreover, all the functions $f_r$, $r\geq 1$, are finite on  interval $I$ since $f_r(x)\leq f_1^r(x)$.
Now we prove that $f_r$ are differentiable.

\begin{lemma} \label{lemma7}The set $I$ is an open interval. For any integer $r\geq 1$,
   function $f_r$ is differentiable on $I$ and satisfies
    $f_r'=r(f_{r-1}+f_{r+1})$.
\end{lemma}
\begin{proof} 
Let  $\eps$ and $x$  be positive numbers.
For any
partition $s\in A_{x+\eps}$ there is a finer partition, $\tilde s$, of
the form $\tilde s=(t,t')\in A_x\times A_\e$,  (because for any
partition of $[0,x+\eps]$ there is a finer one containing $x$). Then  
$$\|R(s)\|\leq \|R(\tilde s)\|\leq f_1(x+\eps)$$
and so $$f_1(x+\e )=\sup_{(t,t')\in A_x\times A_\e}||R(t,t')||.$$ 

Let us consider $(t,t')\in A_x\times A_x$, $t=(t_1,\ldots, t_n)$,
$t'=(t_{n+1},\ldots, t_m)$.
Denoting  $q=(q_1,\ldots,q_{n-1})$ and $q'=(q_n,\ldots,q_{m-1})$  we obtain that
$$R(t,t';q,q')=\left(\c( {t_m})\a (q_{m-1})\cdots\c (t_{n+1})\a
(q_n)\right) (R(t;q)).$$
Since $t\in A_\eps$ by Corollary \ref{cor.5} we have that 
$$\c (t_m)\a (q_{m-1})\cdots\c (t_{n+1})\a (q_n)$$ is given by 
$z\mapsto\frac{az+b}{cz+d}$, where $a,b,c,d\in\RR [q_n,\ldots,q_{m-1}]$
have non-negative coefficients, $||a||=||d||<\cosh\e$,
$||b||=||c||<\sinh\e$, the coefficient of $q_n\cdots q_{m-1}$ in $a$
and the constant term of $d$ are $1$. 
Then  $||a||=1+||a-q_n\cdots q_{m-1}||$ and $||d||=1+||d-1||$ so 
$$||a-q_n\cdots
q_{m-1}||=||a||-1=||d||-1=||d-1||<\cosh\e -1.$$
 Then
$$R(t,t')=\frac{aR(t)+b}{cR(t)+d}=\frac{aR(t)+b}{1+(cR(t)+d-1)}.$$
Since $t\in A_x$ we have $||R(t)||\leq f_1(x)<\infty$. We also have
$\cosh\eps =1+O(\eps^2)$ and $\sinh\eps =O(\eps )$ so
$$aR(t)+b=q_n\cdots q_{m-1}R(t)+(a-q_n\cdots q_{m-1})R(t)+b=q_n\cdots
q_{m-1}R(t)+b+O(\e^2)$$ 
and $$cR(t)+d=1+cR(t)+(d-1)=1+cR(t)+O(\e^2).$$ 
Then for $\eps<\eps_0$, $\eps_0$ small enough depending on $f_1(x)$,
the following holds
\begin{align*}
R(t,t')&=\big(q_n\cdots
q_{m-1}R(t)+b+O(\e^2)\big)\big(1-cR(t)+O(\e^2)\big)\\
&=q_n\cdots
q_{m-1}R(t)+b-q_n\cdots q_{m-1}cR(t)^2+O(\e^2).
\end{align*}
 Since $b-q_n\cdots q_{m-1}cR(t)=O(\e )$ this implies
that for any $r\geq 1$ we have 
\begin{align*}
R(t,t')^r&=(q_n\cdots
q_{m-1})^rR(t)^r+r(q_n\cdots
q_{m-1})^{r-1}R(t)^{r-1}(b-q_n\cdots q_{m-1}cR(t)^2)+O(\e^2)\\
&=(q_n\cdots
q_{m-1})^rR(t)^r+r(q_n\cdots
q_{m-1})^{r-1}(bR(t)^{r-1}-q_n\cdots q_{m-1}cR(t)^{r+1})+O(\e^2).
\end{align*}
 Using that $\sinh\e =\e +O(\e^2)$ it follows that 
\begin{align*}
||R(t,t')^r||&\leq ||R(t)^r||+r||R(t)^{r-1}||\cdot
||b||+r||R(t)^{r+1}||\cdot ||c||+O(\e^2)\\
&\leq f_r(x)+rf_{r-1}(x)\sinh\e
+rf_{r+1}(x)\sinh\e +O(\e^2)\\
&=f_r(x)+r(f_{r-1}(x)+f_{r+1}(x))\e
+O(\e^2). 
\end{align*}
Thus, for  $\eps$ small enough we have
\begin{equation}\label{upper.der}
f_r(x+\e
)\leq
f_r(x)+r(f_{r-1}(x)+f_{r+1}(x))\e +O(\e^2).
\end{equation}
This implies that $I$ is an open interval.

For the reverse inequality we will take $m=n+1$ so $t'$ has dimension
one, $t'=(\e )$, and $q'=q_n$. Then $(t,\e
)\in A_{x+\e}$ so $f_r(x+\e )\geq ||R(t,\e )^r||$. We have
$R(t,\e;q,q_n)=\c (\e )\a (q_n)(R(t,q)),$
i.e. 
$$R(t,\e)=\frac{q_nR(t)+\tanh\e}{q_n\tanh\e R(t)+1}.$$
Since $\tanh\eps =\eps +O(\eps^2)$ by same argument as above we prove
that 
\begin{align*}
R(t,\e )&=q_nR(t)+\eps (1- q_n^2R(t)^2) +O(\e^2).
\end{align*}
and, more generally,
\begin{align*}
R(t,\e )^r&=q_n^rR(t)^r+r\eps q_n^{r-1} ( R(t)^{r-1}- q_n^2R(t)^{r+1}) +O(\e^2).
\end{align*}
%It follows that
%$$f_r(x+\e )\geq ||R(t,\e )^r||=||q_n^rR(t)^r+\e
%(q_n^{r-1}R(t)^{r-1}-q_n^{r+1}R(t)^{r+1})||+O(\e^2).$$ 
Since 
$R(t)^{r-1},R(t)^r,R(t)^{r+1}\in\RR [[q_0,\ldots,q_{n-1}]]$ it follows that
 $$||q_n^rR(t)^r+r\e
(q_n^{r-1}R(t)^{r-1}-q_n^{r+1}R(t)^{r+1})||=||R(t)^r||+r\e
||R(t)^{r-1}||+r\e ||R(t)^{r+1}||.$$
Hence  for any $t\in A_x$ and $\eps$ small enough
% $f_r(x+\e )\geq ||R(t)^r||+r\e
%(||R(t)^{r-1}+||R(t)^{r+1}||)+O(\e^2).$
 \begin{align*}
 f_r(x+\e) &\geq \|R(t,\eps)^r\| \geq
 ||R(t)^r||+r\e
(\|R(t)^{r-1}\|+||R(t)^{r+1}||)+O(\e^2)
\end{align*}
and then 
\begin{equation}\label{lower.der}
f_r(x+\eps)\geq
f_r(x)+\e
r(f_{r-1}(x)+f_{r+1}(x))+O(\e^2).
\end{equation}

Using \eqref{upper.der} and \eqref{lower.der} we obtain 
that  $f_r$ satisfies
$$f_r(x+\eps)=
f_r(x)+\e
r(f_{r-1}(x)+f_{r+1}(x))+O(\e^2).$$
 Thus function $f_r$ is right  differentiable and satisfies
$$\lim_{e\searrow 0}\frac{f_r(x+\e )-f_r(x)}\e=r(f_{r-1}(x)+f_{r+1}(x)).$$
Moreover, since $f_k(x-\eps)\leq f_k(x)\leq f_1^k(x)$, $k\in
\{r-1,r+1\}$, by applying the same argument as in the proof of
\eqref{upper.der} and \eqref{lower.der} to
$x'=x-\eps$
we obtain that 
$$f_r(x)=f_r(x-\eps)+\e
r(f_{r-1}(x-\eps)+f_{r+1}(x-\eps))+O(\e^2)=f_r(x-\eps)+O(\eps)$$
which proves that $f_r$ is also left continuous.

For the left derivative of $f_r$ at $x$ we apply
 the previous analysis to the point   $x'=x-\e$ 
 $$\frac{f_r(x)-f_r(x-\e )}\e=r(f_{r-1}(x-\e )+f_{r+1}(x-\e ))+O(\e
).$$ Since $f_{r-1},f_{r+1}$ are continuous  we obtain 
$$\lim_{\e\searrow 0}\frac{f_r(x)-f_r(x-\e
)}\e=r(f_{r-1}(x)+f_{r+1}(x)).$$
The proof of Lemma \ref{lemma7} is now finished.
 \end{proof}

\begin{theorem}\label{f}
 We have $I=(0,\pi /2)$ and for any $r\geq 1$ function $f_r$  is given by 
 $$f_r(x)=\tan^rx, \quad x\in I.$$
\end{theorem}

\begin{proof} 
We first show that  $[0,\pi/ 2)\subseteq I$ and for any $r\geq 1 $ the following holds
$$f_r(x)\leq\tan^rx,\quad  \forall\  x\in (0,\pi/2).$$
For any $x\in I$ we have $f_1'(x)=f_0(x)+f_2(x)\leq
1+f_1(x)^2$.  Lemma \ref{first.estimate} gives us  that $\lim _{x\rightarrow 0}f_1(x)=0$ and then by integrating the last inequality we obtain 
that $\arctan f_1(x)\leq x$. Thus $f_1(x)\leq\tan x$. For  $r\geq 2$ similar estimates hold since $f_r(x)\leq (f_1(x))^r$.

%If $[0,\frac\pi 2)\not\subseteq I$ then
%$I=[0,T)$ for some $T<\frac\pi 2$. But this implies $\lim_{x\nearrow
%T}f_1(x)\leq\lim_{x\nearrow T}\tan x=\tan T<\infty$, which contradicts
%Lemma \ref{lemma.6}(i).
%For $r\geq 2$ we have $f_r(x)\leq f_1(x)^r\leq\tan^rx$. 
%
%By Corollary \ref{cor} the functions $f_r$ are defined on $[0,\frac\pi
%2)$. In order to remove the incovenience of $f_r$ having only right
%derivative at $0$ we extend the definiton of $f_r$ to $(-\frac\pi 2,\frac\pi
%2)$ by defining $f_r(x)=(-1)^rf_r(-x)$ for $x\in (-\frac\pi 2,0]$. Then
%$f_0\equiv 1$ will hold on $(-\frac\pi 2,\frac\pi 2)$ and for $r\geq
%1$ if $x\in (-\frac\pi 2,0]$ then
%\begin{align*}
%f_r'(x)&=((-1)^rf_r(-x))'=(-1)^{r+1}f_r'(-x)\\
%&=
%(-1)^{r+1}r(f_{r-1}(-x)+f_{r+1}(-x))=r(f_{r-1}(x)+f_{r+1}(x)).
%\end{align*}
% Thus
%the result of Lemma \ref{lemma7} holds for the extended functions $f_r$. The
%inequalities $f_r(x)\leq\tan^rx$ from Corollary \ref{cor} become
%$|f_r(x)|\leq |\tan^rx|$ for the extended functions.

Let $S=\{ (x,y)\in (0,\pi/ 2)\times\RR\mid\, |y\tan x|<1\}$. 
%For
%any $r\geq 0$ the differential of $f_{r+1}(x)y^r$ is
%$(r+1)(f_r(x)+f_{r+2}(x))y^rdx+rf_{r+1}(x)y^{r-1}dy$. 
%For any $(x_0,y_0)\in S$ there are some $M>0$, $0<c<1$ and a
%neighborhood of $U$ such that $|y\tan x|<c$ and $|\tan x|,|y|<M$
%$\forall (x,y)\in U$.  Since $|f_{r+1}(x)y^r|\leq |y^r\tan^{r+1}x|<c^rM$,
%$|(r+1)(f_r(x)+f_{r+2}(x))y^r|\leq
%(r+1)(|\tan^rx|+|\tan^{r+2}x|)|y|^r<(r+1)(c^r+c^rM^2)$ and
%$|rf_{r+1}(x)y^{r-1}|\leq (r+1)|y^{r-1}\tan^{r+1}x|<c^{r-1}M^2$. It
%follows that the serieses $\sum_rf_{r+1}(x)y^r$,
%$\sum_r(r+1)(f_r(x)+f_{r+2}(x))y^r$ and $\sum_rrf_{r+1}(x)y^{r-1}$ are
%uniformly absolutely convergent. 
%
Using the properties of functions $f_r$ above we get that 
 function $g$ defined by
$$g(x,y)=f_1(x)+f_2(x)y+f_3(x)y^2+\cdots$$ is well defined and
differentiable on $S$. Explicit computations show that $g$ satisfies
the following first order equation
$$(y^2+1)g_y(x,y)-g_x(x,y)=-2yg(x,y)-1, \quad \forall \ (x,y)\in S.$$ 
In order to solve it we need some boundary conditions. Observe that 
since $f_r(x)\leq \tan ^rx$ for all $ r\geq 1$, function $g$ satisfies
$\lim _{x\rightarrow 0} g(x,y)=0$ for all $y\in\RR$. Solving the above
equation by the method of characteristics we obtain that
$$g(x,y)=\frac{\tan
x}{1-y\tan x}.
$$

Developing in $y$ power series we get that 
 $f_r(x)=\tan^rx$, as claimed. In particular,
$f_1(x)=\tan x$. 
Since $f_1$ is increasing we also obtain that $I=(0,\pi/2)$ and the proof is finished. 
\end{proof}

\subsection{Proofs of the two Lemmas}
We are now able to prove  Lemma \ref{imp} and Lemma \ref{negative}.

\begin{proof}[Proof of Lemma \ref{imp}]
By definition of $E$ and Theorem \ref{f}
we have that 
\begin{align*}
\|Q_{n+1}\|_{AP}&=\|E(|d_1|,\ldots, |d_n|)\|=\|R(\arctanh |d_1|,\ldots, \arctanh |d_n|)\|\\
&\leq f_1\Big(\sum _{k=1}^n \arctanh |d_k|\Big)=\tan \Big (\sum _{k=1}^n \arctanh |d_k|\Big).
\end{align*}
Then estimate \eqref{est.ap} follows.
\end{proof}

\begin{proof}[Proof of Lemma \ref{negative}]
Let $\a >0$. By Theorem \ref{f}, $f_1(\a )=\tan\a$ if $\a <\pi/2$ and
$f_1(\a )=\infty$ if $\a\geq\pi/2$. This means that for any $\e,N>0$
there exists a fine enough partition $t=(t_1,\ldots, t_{n-1})\in
A_\alpha$ such that $\|R(t)\|\geq\tan\a -\e$ if $\a <\pi/2$ and
$\|R(t)\|\geq N$ if $\a\geq\pi/2$. Choosing, if necessary, finer
partitions we can assume that $t_k<\arctanh d$, $k=1,\ldots, n-1$.

Choosing $\tilde d_k=\tanh t_k$, $k=1,\ldots, n-1$, we have $0< \tilde
d_k\leq d$, $\sum_{k=1}^{n-1}\arctanh\tilde
d_k=\sum_{k=1}^{n-1}t_k=\a$ and 
$$\|E(\tilde d_1,\ldots, \tilde d_{n-1})\|=\|R(t)\|\geq\tan\a
-\e\quad\text{or}\quad N,$$ 
corresponding to $\a <\pi/2$ or $\a\geq\pi/2$, respectively.

Then for any $\{c_k\}_{k=1}^{n-1}$ linearly independent over $\QQ$ and
any $\{d_k\}_{k=1}^{n-1}$ with $|d_k|=\tilde d_k$ we have by Lemma
\ref{ineg.q.E} that $\{Q_k\}_{k=1}^n$ defined in \eqref{def.q}
satisfies
$$\|Q_n\|_{AP}=\|E(|d_1|,\ldots,
|d_{n-1}|)\|=\|E(\tilde d_1,\ldots, \tilde d_{n-1})\|\geq\tan\a
-\e\quad\text{or}\quad N,$$
accordingly, and the proof is finished.
\end{proof}

%
%\begin{corollary} (i) For any $0<x<\frac\pi 2$ and any $\e >0$ there is some $0<d<1$
%such that for any $d_1,\ldots,d_n$ with $0<d_j<d$ and $d_1+\cdots
%+d_n\leq x$ we have $||E(d_1,\ldots,d_n)||<\tan x+\e$.
%
%(ii) For any $0<d<1$ and any $M>0$ there is a sequence
%$d_1,\ldots,d_n$ with $0<d_j<d$ and $d_1+\cdots +d_n<\frac\pi 2$ such that
%$||E(d_1,\ldots,d_n)||>M$.
%\end{corollary}
%\begin{proof} (i) We have $\tan x+\e=\tan (x+\e')$ for some $\e'>0$ with
%$x+\e'<\frac\pi 2$. Since $\lim_{t\to 0}\frac {{\rm arctan}\,u}u=1$ there
%is some $0<d<1$ such that $\frac {{\rm arctan}\,u}u<\frac{x+\e'}x$ for
%$0<u<d$. If $0<d_j<d$ and $d_1+\cdots +d_n\leq x$ then
%\begin{align*}
%||E(d_1,\ldots,d_n)||&=||R({\rm arctanh}\,d_1,\ldots,{\rm
%arctanh}\,d_n)||\leq f_1({\rm arctanh}\,d_1+\cdots +{\rm
%arctanh}\,d_n)\\
%&=\tan ({\rm arctanh}\,d_1+\cdots +{\rm arctanh}\,d_n).
%\end{align*}
% But
%$0<d_j<d$ so $\frac {{\rm arctan}\,d_j}{d_j}<\frac{x+\e'}x$, which implies that
%$${\rm arctan}\,d_1+\cdots +{\rm arctan}\,d_n<\frac{x+\e'}x(d_1+\cdots
%+d_n)\leq x+\e',$$ which implies $\tan ({\rm arctan}\,d_1+\cdots +{\rm
%arctan}\,d_n)<\tan (x+\e')=\tan x+\e$.
%
%(ii) We have $\lim_{t\in A_{\frac\pi 2}}||R(t)||=f_1(\frac\pi 2)=\infty$ so
%for a fine enough partition $t\in A_{\frac\pi 2}$ we have
%$||R(t)||>M$. Let $t=(t_1,\ldots,t_n)$ be such partition which also
%satisfies $t_j<d$. If $d_j=\tanh t_j$ then $$0<d_j<t_j<d,\quad
%d_1+\cdots +d_n<t_1+\cdots +t_n<\frac\pi 2$$ and
%$$||E(d_1,\ldots,d_n)||=||R(t_1,\ldots,t_n)||>M.$$ \end{proof}

{\bf Acknowledgments.} L.I. was partially supported by Grant PN-II-ID-PCE-2011-3-0075 of the Romanian National Authority for Scientific Research, CNCS -- UEFISCDI, by Grant MTM2011-29306-C02-00, MICINN, Spain, ERC Advanced Grant FP7-246775 NUMERIWAVES and FA9550-15-1-0027 of AFOSR.

%
%E.Z. 
%was partially supported by Grant MTM2011-29306-C02-00, MICINN, Spain, ERC Advanced Grant FP7-246775 NUMERIWAVES, ESF Research Networking Programme OPTPDE and Grant PI2010-04 of the Basque Government.

E.Z.  was supported by the Advanced Grant FP7-246775 NUMERIWAVES of the European Research Council Executive Agency, FA9550-14-1-0214 of the EOARD-AFOSR, FA9550-15-1-0027 of AFOSR, the BERC 2014-2017 program of the Basque Government, the MTM2011-29306-C02-00 and SEV-2013-0323 Grants of the MINECO and a Humboldt Award at the University of Erlangen-N\"urnberg.

%
%
%%
%\bibliographystyle{plain}
%\bibliography{habbiblio}

\begin{thebibliography}{1}

\bibitem{MR2049025}
V.~Banica.
\newblock Dispersion and {S}trichartz inequalities for {S}chr\"odinger
  equations with singular coefficients.
\newblock {\em SIAM J. Math. Anal.}, 35(4):868--883 (electronic), 2003.

\bibitem{MR2227135}
N.~Burq and F.~ Planchon.
\newblock Smoothing and dispersive estimates for 1{D} {S}chr{\"o}dinger
  equations with {BV} coefficients and applications.
\newblock {\em J. Funct. Anal.}, 236(1):265--298, 2006.

\bibitem{MR2460203}
C.~Corduneanu.
\newblock {\em Almost periodic oscillations and waves}.
\newblock Springer, New York, 2009.

\bibitem{MR2794904}
L.~ Demanet and G.~Peyr{\'e}.
\newblock Compressive wave computation.
\newblock {\em Found. Comput. Math.}, 11(3):257--303, 2011.

\bibitem{MR2327824}
J-P.~ Fouque, J.~Garnier, G.~Papanicolaou, and K.~S{\o}lna.
\newblock {\em Wave propagation and time reversal in randomly layered media},
  volume~56 of {\em Stochastic Modelling and Applied Probability}.
\newblock Springer, New York, 2007.

\bibitem{MR2445437}
Loukas Grafakos.
\newblock {\em {Classical Fourier Analysis}}, volume 249 of {\em Graduate
  Texts in Mathematics}.
\newblock Springer, New York, second edition, 2008.


\end{thebibliography}
%%

\end{document}